\documentclass[%
 reprint,
 amsmath,amssymb,
 aps,
]{revtex4-2}

\usepackage{graphicx}
\usepackage{dcolumn}
\usepackage{bm}

\usepackage{amsmath, amsthm, amssymb}
\usepackage{color}

\newtheorem{claim}{Claim}
\theoremstyle{definition}

\theoremstyle{remark}

\newcommand{\cov}{\text{cov}}
\newcommand{\var}{\text{var}}
\newcommand{\R}{{\mathbb R}}

\newcommand{\vvec}[1]{{\bf{#1}}}

\begin{document}

\preprint{APS/123-QED}

\title{High-dimensional dynamics in low-dimensional networks}

\author{Yue Wan}
\author{Robert Rosenbaum}
 \affiliation{Department of Applied and Computational Mathematics and Statistics, University of Notre Dame, Notre Dame, IN 46556, USA.}

 \email{Robert.Rosenbaum@nd.edu}

\date{\today}

\begin{abstract}
Many networks in nature and  applications have an approximate low-rank structure in the sense that their connectivity structure is dominated by a few dimensions. It is natural to expect that dynamics on such networks would also be low-dimensional. Indeed,  theoretical results show that low-rank networks produce low-dimensional dynamics whenever the network is isolated from external perturbations or input. However, networks in nature are rarely isolated. Here, we study the  dimensionality of dynamics in recurrent networks with low-dimensional structure driven by high-dimensional inputs or perturbations. We find that dynamics in such networks can be high- or low-dimensional and we derive mathematical conditions on the network structure under which linearized dynamics are high-dimensional. In many low-rank networks, dynamics are suppressed in directions aligned with the network's low-rank structure, a phenomenon we term ``low-rank suppression.'' We show that several low-rank network structures arising in nature satisfy the conditions for generating high-dimensional dynamics and low-rank suppression. Our results clarify important, but counterintuitive relationships between a recurrent network's connectivity structure and the structure of {its response to external input.}
\end{abstract}

\maketitle


\section{Introduction}

Recent work shows that many networks arising in nature and  applications have low-rank structure~\cite{thibeault2024low}. 
This observation raises an important question: 
What is the relationship between the dimensionality of a network's \textit{structure} and the dimensionality of the \textit{dynamics} on the network?

In Neuroscience, for example, the implications of low-rank recurrent connectivity on neural dynamics is a topic of intense research~\cite{aljadeff2016low,mastrogiuseppe2018linking,landau2021macroscopic,dubreuil2022role,landau2018coherent,schuessler2020dynamics,valente2022extracting,cimevsa2023geometry,beiran2023parametric,mastrogiuseppe2025stochastic,aljadeff2016low,clark2025connectivity}. 
Theoretical studies demonstrating low-dimensional dynamics in low-rank neuronal network models are consistent with some neural recordings showing low-dimensional neural activity~\cite{kaufman2014cortical,gallego2017neural,stringer2019spontaneous}.  
However, many of these recordings are made in the context of low-dimensional stimuli or behavior. A growing number of more recent studies have shown that neural activity can be high-dimensional, particularly in response to high-dimensional stimuli or behavior~\cite{stringer2019high,stringer2019spontaneous,lanore2021cerebellar,avitan2022not,markanday2023multidimensional,manley2024simultaneous}.  These observations raise the important point that the dimensionality of dynamics on a network depends on the structure of the network's external input in addition to its internal structure.

{There is a large body of theoretical work on  the dimensionality of dynamics on low-rank networks~\cite{mante2013context,sussillo2013opening,aljadeff2016low,landau2018coherent,mastrogiuseppe2018linking,schuessler2020dynamics,dubreuil2022role,valente2022extracting,cimevsa2023geometry,beiran2023parametric,thibeault2024low,landau2018coherent,huang2019circuit,landau2021macroscopic,o2022direct,beiran2023parametric,cimevsa2023geometry,schuessler2024aligned,mastrogiuseppe2025stochastic,clark2025connectivity,aljadeff2016low}.}
Some theoretical and computational studies demonstrating low-dimensional dynamics in low-rank networks focus on networks that are self-contained in the sense that they operate in the absence of noise or external inputs or perturbations~\cite{aljadeff2016low,landau2018coherent,thibeault2024low}. Other work considers external input, but assumes that the input is perfectly aligned with the network's low-dimensional structure~\cite{landau2021macroscopic}.  
In reality, networks in nature are rarely isolated from internal or external perturbations or inputs, and these perturbations are not necessarily aligned with the network's low-dimensional structure. 
Some~\cite{mastrogiuseppe2018linking,dubreuil2022role} considers the case of general external input to networks that are weakly low-rank in the sense that the largest singular values are $\mathcal O(1)$ in magnitude, consistent with connection weights that scale like $1/N$ where network size, $N$, is assumed to be large.

In this study, we consider networks with strongly low-rank structure in the sense that they have a small number of asymptotically large singular values, while the rest are small or moderate in magnitude. This is consistent with connection weights that are asymptotically larger than $1/N$ in magnitude, often taken to be $\mathcal O(1/\sqrt N)$ in theoretical work. 
Importantly, we study the dimensionality of the dynamics on these networks when they are driven by external input. The dimensionality of activity driven by intrinsic dynamics is a fundamentally different question treated in previous work~\cite{thibeault2024low,landau2018coherent,aljadeff2016low}. We also focus on dynamics near linearly stable equilibrium. We first show by example that  networks under these conditions can produce high- or low-dimensional dynamics in response to high-dimensional inputs, depending on the network's structure. Perhaps counterintuitively, we also find that networks under these conditions \textit{suppress} inputs aligned to their dominant low-dimensional structure, an effect we call ``low-rank suppression.''

We use mathematical analysis to derive  conditions under which low-rank suppression and high-dimensional dynamics can occur in the linearized dynamics of strongly low-rank networks close to a stable equilibrium. Our analysis shows that low-rank suppression arises generally in these networks. High-dimensional responses to high-dimensional external inputs, on the other hand, arise when the dominant low-rank part of the network's connectivity matrix has a non-vanishing component with the EP property. EP matrices are matrices for which the column space and row space are equal (the abbreviation stands for ``equal principal'' or ``equal projectors'')~\cite{schwerdtfeger1950introduction,bernstein2018scalar}. 
Low-dimensional dynamics can arise in the linearized dynamics of strongly low-rank networks through a form of non-normal amplification~\cite{trefethen1991pseudospectra,murphy2009balanced,hennequin2012non}, but our mathematical analysis shows that this form of non-normal amplification can only occur when the network is non-EP. Non-normality is not sufficient. To our knowledge, this connection between the EP property of matrices, non-normal amplification, and the dimensionality of recurrent network dynamics has not been described in previous work.



We show that many common structural features such as biased weights, modularity, and spatial connectivity structure naturally produce strongly low-rank connectivity matrices with low-rank parts satisfying the EP condition, implying that these networks produce low-rank suppression and high-dimensional responses to high-dimensional inputs. We also draw connections between low-rank suppression and the mathematical theory of  balanced networks~\cite{van1996chaos}, extending previous work in this direction~\cite{landau2021macroscopic}. 
Finally, we demonstrate our conclusions in dynamics on a real  epidemiological network.

Our conclusions have important implications for the interpretation of low-dimensional network structure. In neuroscience, our results can 
explain why neural populations generate high-dimensional responses to high-dimensional stimuli and tasks~\cite{stringer2019high}. Our results also generalize and extend the theory of excitatory-inhibitory balance~\cite{van1996chaos,landau2021macroscopic}, and amplification arising from breaks in this balance~\cite{landau2016impact,pyle2016highly,ebsch2018imbalanced}. 
Beyond neuroscience, our results imply that perturbations misaligned to low-rank network structure are most effective at driving responses. This counterintuitive observation could be used to design more effective interventions to epidemiological, biological, social, and other networks.

\section{High-dimensional dynamics and low-rank suppression in a network with rank-one structure}

 \begin{figure*}
 \centering{
 \includegraphics[width=6.5in]{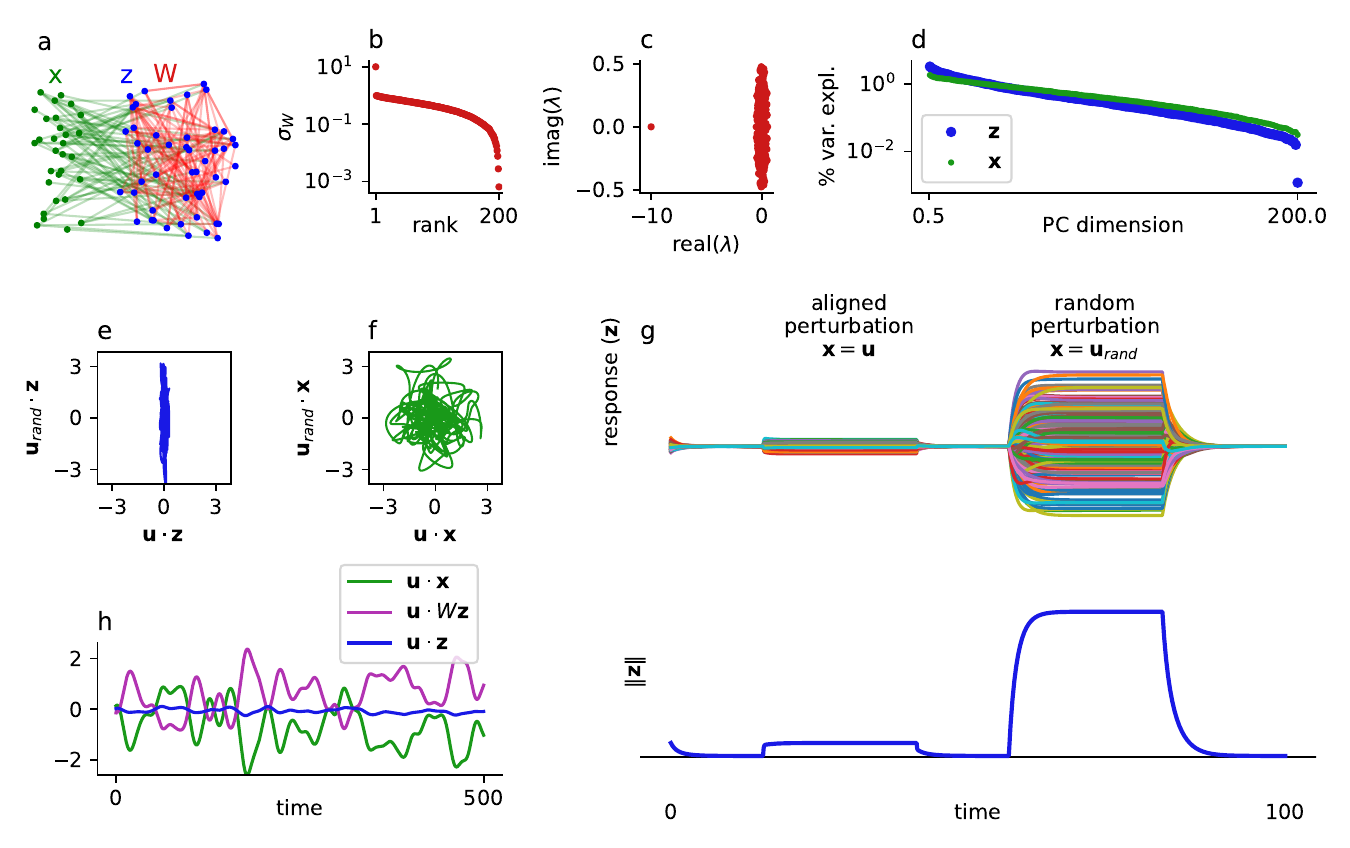}
 }
 \caption{{\bf Response properties of a recurrent network with rank-one structure.}  {\bf a)} Schematic of model: The connectivity matrix, $W$, quantifies connections between nodes, $\vvec z$, which receive external perturbations or input, $\vvec x$.   {\bf b)} The singular values of  $W$ have one dominant term, indicating approximate rank-one structure. {\bf c)} The eigenvalues of $W$ have a dominant, negative term. 
 {\bf d)} The distribution of variance across principal components of a Gaussian stochastic input ($\vvec x(t)$; green) and the response ($\vvec z(t)$; blue). {\bf e,f)} The network response (e) and input (f) projected onto the plane determined by $\vvec u$ and a random vector, $\vvec u_\textrm{rand}$, demonstrates low-rank suppression along $\vvec u$. 
 {\bf g)} The network response ($\vvec z(t)$, top) and its norm ($\|\vvec z(t)\|$, bottom) given an input aligned to the low-rank structure of the network ($\vvec x_\textrm{aligned}$) and a random input ($\vvec x_\textrm{rand}$). {The norm of $\vvec z$ is more than 11 times smaller in response to the aligned versus random perturbation.} 
 {\bf h)} Local network input (purple) cancels with external input (green) to produce suppressed network responses (blue) in the direction of $\vvec u$.
 }
 \label{F1}
 \end{figure*}

For illustrative purposes, we begin with a simple, linear recurrent network model (Figure~\ref{F1}a) 
\begin{equation}\label{Edzdt1}
\begin{aligned}
\tau \frac{d\vvec z}{dt}&=-\vvec z+W\vvec z+\vvec x.
 \end{aligned}
\end{equation}
where $\vvec x(t)$ is an external perturbation or input, $\vvec z(t)$ is the network response, and $\tau=1$. The $N\times N$ recurrent connectivity matrix, $W$, takes the form 
\begin{equation*}
W= W_0+W_1
\end{equation*}
where 
\[
W_0=c\vvec u \vvec u^T
\]
is a normal, rank-one matrix
and
\begin{equation*}
W_1=\frac{\rho}{\sqrt N}Z
\end{equation*}
is a full rank random matrix. 
Specifically, $\vvec u$ is a random vector with $\|\vvec u\|=1$, $Z$ is an $N\times N$ matrix with entries drawn i.i.d.~from a standard normal distribution, and $\rho>0$. If 
\[
|c|\gg \rho
\]
then $W$ is ``effectively low-rank''~\cite{thibeault2024low} in the sense that it has one large singular value near $|c|$ and the remaining singular values are much smaller (approximately bounded by $2\rho$; Figure~\ref{F1}b). 
Similarly, $W$ has one large eigenvalue near $c$ and the remaining eigenvalues lie approximately within a circle of radius $\rho$ in the complex plane (Figure~\ref{F1}c)~\cite{girko1985circular,tao2010random}.

Importantly, stability of the network dynamics requires that $\rho<1$ and $c<1$~\cite{sompolinsky1988chaos}. If the dynamics were unstable, then $\|\vvec z(t)\|$ would grow exponentially toward $\infty$, so we only consider parameter regimes with stable dynamics. 
Since strong low-rank structure also requires $|c|\gg\rho$, stability requires $c<0$ when $\rho=\mathcal O(1)$. In simulations here, we take $\rho=0.5$ and $c=-10$. We consider more general models and network structures later.

We first provided a high-dimensional input, $\vvec x(t)$. Specifically, for our first simulation, each $\vvec x_j(t)$  was an i.i.d., smooth, stationary Gaussian process, which  models internal noise, external perturbations, or external input to the network. Conventional wisdom might lead us to expect a low-dimensional network response dominated by variability in the direction of $\vvec u$. Indeed, a \textit{feedforward} network with the same connectivity matrix and same input produces approximately one-dimensional dynamics because it amplifies inputs aligned to $\vvec u$ (Figure~\ref{SuppFigFfwd} in Appendix~\ref{AppSuppFigures}). 
However, this conclusion does not necessarily carry over to  recurrent networks. 

In the recurrent network, the variance explained by the principal components of $\vvec z(t)$ decayed similarly to those of $\vvec x(t)$ (Figure~\ref{F1}d), indicating that $\vvec z(t)$ was high-dimensional like $\vvec x(t)$.  The only exception was the principal component that explained the \emph{least} variance in $\vvec z(t)$, which was much weaker than the other principal components (last blue dot in Figure~\ref{F1}d). 
Perhaps surprisingly, this weakest principal component direction was closely aligned to $\vvec u$ (the angle was less than $8^\circ$). Consistent with this finding, the variance of $\vvec z(t)$ in the direction of $\vvec u$ was more than 132 times \textit{smaller} than the variance of $\vvec z(t)$ along a random direction (Figure~\ref{F1}e) even though  the variability of $\vvec x(t)$ was similar in each direction (Figure~\ref{F1}f).

To better understand these results, we next simulated the same network with two different external input patterns. One pattern was  aligned with the low-dimensional structure of $W$,
\[
\vvec x_\textrm{aligned}= \vvec u,
\]
while the other had a random direction,  
\[
\vvec x_\textrm{random}=\vvec u_\textrm{rand},
\]
which was generated identically to, but independently from $\vvec u$. 
Intuitively, we might expect the network to respond more strongly to the aligned input than to the random input, as in a feedforward network (Figure~\ref{SuppFigFfwd} in  in Appendix~\ref{AppSuppFigures}). In reality, we observed exactly the opposite (Figure~\ref{F1}g): The response to $\vvec x_\textrm{random}$ was more than 11 times {larger} than the response to $\vvec x_\textrm{aligned}$.

We use the term ``low-rank suppression'' to refer to this phenomenon in which inputs aligned to the low-rank structure of a network are suppressed by the network's dynamics. In the absence of other directions that are amplified by the network, low-rank suppression leads to high-dimensional responses to high-dimensional inputs (as in Figure~\ref{F1}d).

The results in Figure~\ref{F1} contrast with some previous modeling work demonstrating low-dimensional dynamics in low-rank recurrent networks~\cite{mastrogiuseppe2018linking,landau2018coherent,landau2021macroscopic,dubreuil2022role,thibeault2024low}. In Discussion and in Appendix~\ref{SuppComparison}, we provide a more detailed review of this previous work and its relation to ours.

On closer inspection, low-rank suppression in this example is not surprising because the Jacobian at the fixed point has a large negative eigenvalue (near $c-1$ with $c<0$ and $|c|\gg 0$) with associated eigenvector near $\vvec u$ so the dynamics in Eq.~\eqref{Edzdt1} are highly compressive along $\vvec u$. However, as we will show in the next section, low-rank suppression occurs even in the absence of a large, negative eigenvalue.

A second explanation for low-rank suppression generalizes to examples without a  large, negative eigenvalue: Consider the steady-state-solution to Eq.~\eqref{Edzdt1}, 
\begin{equation}\label{EzSS}
\vvec z=[I-W]^{-1}\vvec x.
\end{equation}
Convergence to the steady-state requires that $\vvec x(t)=\vvec x$ is static, but as long as $\vvec x(t)$ varies more slowly than $\tau=1$, solutions approximately track the quasi-steady state given by Eq.~\eqref{EzSS}. 
Because of the matrix inverse in Eq.~\eqref{EzSS}, the \textit{large} singular value of $W$ near $|c|$ produces a \textit{small} singular value of $[I-W]^{-1}$ near $1/(1-|c|)\approx 1/|c|$ with left and right singular vectors near $\vvec u$. Hence, the recurrent network suppresses inputs in the direction of $\vvec u$. 

Both of the explanations above rely on the assumption that $W_0$ is a symmetric matrix, or at least a normal matrix, but in the next section, we will next consider a case in which $W_0$ is non-normal and non-symmetric.

An interesting consequence of low-rank suppression is that external perturbations cancel nearly perfectly with  recurrent inputs in the direction of $\vvec u$. More precisely, note that $W\vvec z(t)$ in Eq.~\eqref{Edzdt1} can be interpreted as a vector of internal input to each node, whereas $\vvec x(t)$ is external input and $\vvec z(t)$ is the network response. 
In the quasi-steady state, we have
\begin{equation}\label{EzFP}
\vvec z\approx W\vvec z+\vvec x.
\end{equation}
In other words, $\vvec z$ tracks the sum of internal and external inputs in the quasi-steady state.

Since $W$ has a large singular value with left and right singular values aligned to $\vvec u$, multiplication by $W$ amplifies the direction $\vvec u$. Therefore, the internal input, $W\vvec z$, is large in the direction of $\vvec u$ whenever $\vvec z$ is moderate in the direction of $\vvec u$. In other words, $|\vvec u\cdot W\vvec z|\gg |\vvec u \cdot \vvec z|$. This fact might appear to present a paradox because the direction $\vvec u$ is \textit{amplified} in the product $W\vvec z$, but $W\vvec z$ is one component of $\vvec z$ (Eq.~\eqref{EzFP}) and the direction $\vvec u$ is \textit{suppressed} in $\vvec z$, as we saw in Figure~\ref{F1}e,g.



This apparent paradox is resolved by a cancellation between $\vvec W\vvec z$ and $\vvec x$ in the direction of $\vvec z$. Specifically, 
under low-rank suppression, internal input cancels nearly perfectly with external input in the direction of $\vvec u$ (\textit{i.e.}, $\vvec u\cdot W\vvec z\approx -\vvec u\cdot \vvec x$) so that the response, $\vvec z$, is weak in the direction of $\vvec u$ (Figure~\ref{F1}h). We refer to this effect as ``low-rank cancellation,'' which is closely related to the theory of excitatory-inhibitory balance in neural circuits~\cite{van1996chaos,van1998chaotic} as we will show later.

\section{Low-dimensional dynamics and low-rank suppression in a non-normal network with rank-one structure}

 \begin{figure*}
 \centering{
 \includegraphics[width=6.5in]{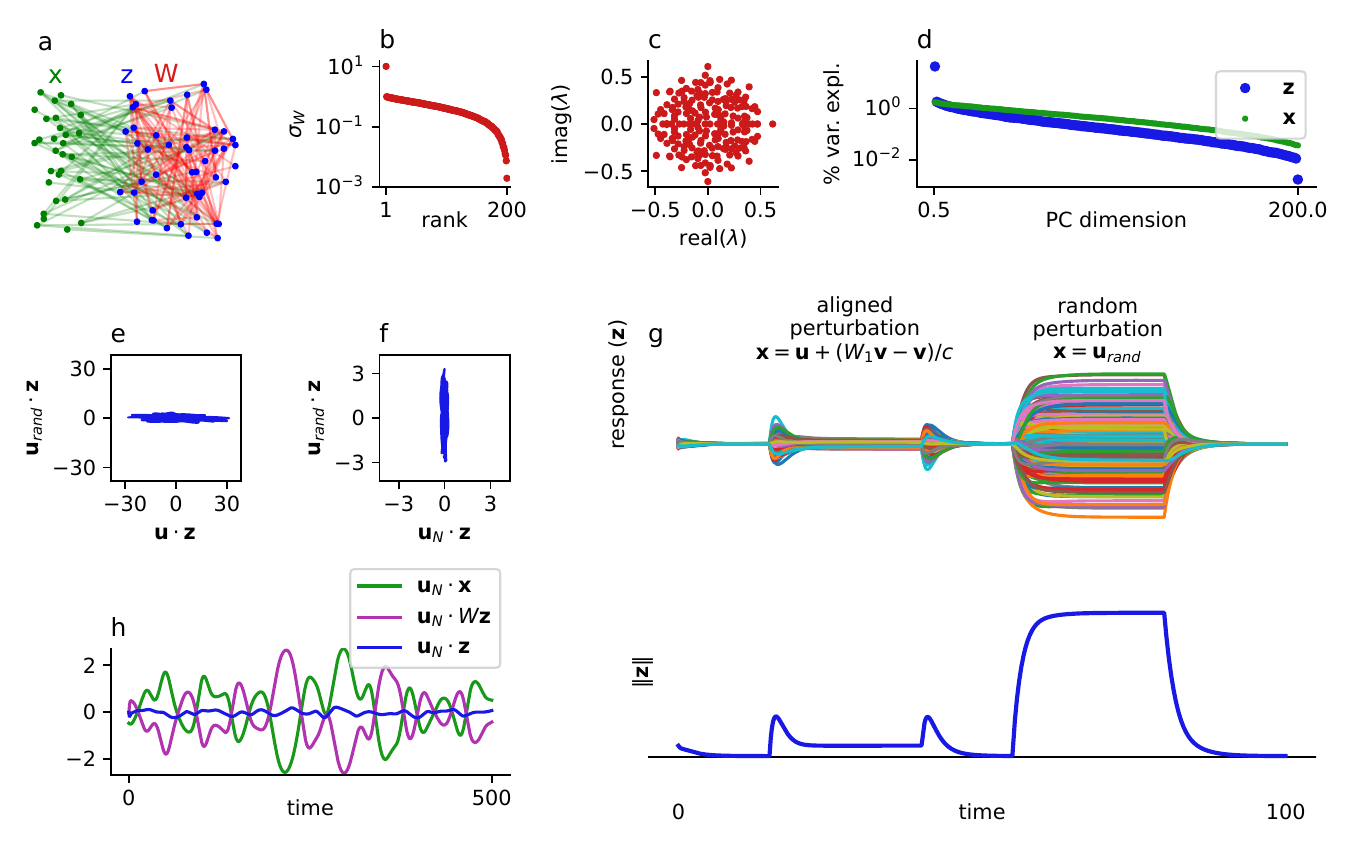}
 }
 \caption{{\bf Response properties of a non-normal recurrent network with rank-one structure.}  Same as Figure~\ref{F1} except the low-rank part of $W$ is non-normal. The network still exhibits low-rank suppression and cancellation (f, g, h) even in the absence of strong negative eigenvalues (c). The network exhibits low-dimensional dynamics (d), in contrast to Figure~\ref{F1}. {The norm of $\vvec z$ is more than 13 times smaller in response to the aligned versus random perturbation.} 
 }
 \label{F2}
 \end{figure*}

Our observation of low-rank suppression in the example above might seem unsurprising due to the presence of a large, negative eigenvalue. It is widely known that an eigenvalue with large, negative real part produces a suppressive dynamic. 
Likewise, our observation of high-dimensional dynamics might seem like a simple consequence of high-dimensional input. 
We next address these points with an example in which $W_0$ is highly non-normal (however, we later show that the EP property, not normality, is key to determining the dimensionality of dynamics). In this example, low-rank suppression arises in the absence of a large negative eigenvalue and the network's response is low-dimensional even though the input is high-dimensional.

We specifically consider an example in which the model and all parameters are the same as in Figure~\ref{F1} except that 
\[
W_0=c\vvec u\vvec v^T
\]
where $\vvec u$ and $\vvec v$ are random, orthogonal vectors with $\|\vvec u\|=\|\vvec v\|=1$. 
This matrix is highly non-normal in the sense that its left  and right singular vectors are orthogonal, and its non-zero eigenvalue does not have a simple relationship to its singular values. 

The distribution of singular values of $W=W_0+W_1$ are identical to those from our previous example (Figure~\ref{F2}b; compare to Figure~\ref{F1}b), but the eigenvalues are different (Figure~\ref{F2}c; compare to Figure~\ref{F1}c). Specifically, the eigenvalues no longer have a dominant negative outlier. They are all $\mathcal O(1)$ in magnitude and clustered in a single bulk.

Repeating the simulations from above for this network, we found that the response to high-dimensional input was low-dimensional in the sense that variability was much greater in one direction than all others: Over 40\% of the variance in $\vvec z(t)$ was captured by a single principle component (first blue dot in Figure~\ref{F2}d), compared to less than 2\% for any orthogonal direction (the remaining blue dots). This dominant direction was partially aligned to $\vvec u$ (the angle between the dominant direction and $\vvec u$ was $31^\circ$). Consistent with this observation, variability was $100$ times larger in the direction of $\vvec u$ than in a random direction (Figure~\ref{F2}e). Hence, low-dimensional dynamics were produced by high-dimensional input.


In addition, there was a direction in which variability was especially weak (last blue dot in Figure~\ref{F2}d), similar to our previous example. This last principal component direction, $\vvec u_N$, was partially aligned to $\vvec v$ (the angle between $\vvec u_N$ and $\vvec v$ was $13^\circ$). While the variability along $\vvec v$ itself was not much smaller than a random direction (due to imperfect alignment between $\vvec u_N$ and $\vvec v$), the variability along $\vvec u_N$ was 190 times smaller than the variability in a random direction (Figure~\ref{F2}f). 
Also, low-rank cancellation was observed in the direction of $\vvec u_N$ (Figure~\ref{F2}h). 

Low-rank suppression was observed in the sense that an input in the direction of $\vvec x_\textrm{aligned}=\vvec u+(W_1\vvec v-\vvec v)/c$ produced a much weaker response than input in a random direction (Figure~\ref{F2}g). The choice of $\vvec x_\textrm{aligned}$ will be discussed in Section~\ref{SLDD} and Appendix~\ref{SStochasticAnalysis}, but note that $\vvec x_\textrm{aligned}$ is closely aligned to $\vvec u$ whenever $c$ is large. Hence, low-rank suppression and low-rank cancellation arise even in the absence of a dominant negative eigenvalue.

To summarize our observations from the two examples above (Figures~\ref{F1} and \ref{F2}), both networks produced low-rank suppression and low-rank cancellation, but only the network with a non-normal connectivity matrix (Figure~\ref{F2}) produced low-dimensional dynamics. In the next section, we derive more precise conditions for low-rank suppression and low-dimensional dynamics.




\section{General conditions for low-rank suppression and high-dimensional dynamics}

Above, we considered a simple network with rank-one structure. Here, we consider more general classes of networks, specify a list of assumptions that we make about the networks, and  precise conditions for low-rank suppression and low-dimensional dynamics under these assumptions. {We first state two assumptions that we make throughout this section. Then, we show that low-rank suppression occurs under  these assumptions alone. Finally, we analyze the conditions for high-dimensional dynamics under various assumptions on $\vvec x(t)$. 
}

\subsection{Assumption 1: Dynamics are linear or near a stable equilibrium.}\label{SLin}

We consider a general class of dynamics of the form
\begin{equation}\label{EzDS}
\tau \frac{d\vvec z}{dt}=-\vvec z+\vvec F(\vvec z,\vvec x)
\end{equation}
for some smooth vector field, $\vvec F:\R^{N}\times \R^N\to\R^N$, where $\tau>0$ sets the timescale of the dynamics. 
The dependence of $\vvec F(\vvec z,\boldsymbol x)$ on $\vvec z$ defines interactions  between nodes, but it is difficult to define a fixed network of interactions, $W$, from the general dynamics in Eq.~\eqref{EzDS}. 
{In nonlinear systems, the effective connectivity between nodes can change with the state of the system. Moreover, nonlinear recurrent networks can produce any dynamics (high- or low-dimensional) in response to any input~\cite{maass2002real,jaeger2004harnessing}, so it is difficult to make any general statements about the dimensionality and structure of dynamics in recurrent networks far from a stable equilibrium. 
To resolve this difficulty, we assume that dynamics represent a small deviation away from a stable equilibrium so that we can linearize them.  
A change of coordinates under this linearization gives Eq.~\eqref{Edzdt1} where $W$ is defined as the Jacobian of $\vvec F$ at the fixed point ($W=\partial_{\vvec z}\vvec F(\vvec z_0,\vvec x_0)$; see Appendix~\ref{SuppLin} for details). 
Our assumption of stability guarantees that all eigenvalues of $W$ have real part less than 1. To avoid systems close to instability, we additionally assume that the eigenvalues are bounded away from 1, in other words $W$ does not have any eigenvalues close to $1$. For our examples, this implies that the spectral radius, $\rho$, of the random part, $W_1$, is not close to $1$.  
Eigenvalues near $1$ or eigenvalues larger than $1$ can produce low-dimensional dynamics through a different mechanism than the mechanisms that we are studying~\cite{sussillo2013opening,mante2013context,aljadeff2016low,huang2019circuit,thibeault2024low} (see Figures~\ref{SuppFigUnstable}--\ref{SuppFigCritical} in Appendix~\ref{AppSuppFigures} for examples).}

\subsection{Assumption 2: The connectivity matrix is strongly low rank in an asymptotic sense.}\label{SstronglyLR}

{We additionally assume that the network, $W$, has a strongly low-rank structure. To make this assumption precise, we consider networks in the limit of large $N$ and assume that a small number of singular values diverge with $N$ while the rest remain $\mathcal O(1)$ or smaller (see Appendix~\ref{SuppStrong} for a more precise statement).}

{This assumption has two components. First, that the connectivity matrix has a small number of dominant singular values. In later sections, we show that this property arises naturally in several natural network structures such as networks with biased weights, modular structure, or smooth spatial structure.}

{
Secondly, this assumption implies that the dominant singular values grow  large with $N$. 
This contrasts to other work on ``weakly'' low-rank networks in which the largest singular values are $\mathcal O(1)$~\cite{mastrogiuseppe2018linking,schuessler2020dynamics,dubreuil2022role,valente2022extracting,cimevsa2023geometry,beiran2023parametric}. Weakly low-rank structure arises naturally when mean weights scale like $\mathcal O(1/N)$ whereas strongly low-rank structure arises when mean weights are asymptotically larger than $1/N$, for example when they scale like $\mathcal O(1/\sqrt N)$ as in balanced networks~\cite{van1996chaos,renart2010asynchronous,rosenbaum2017spatial,baker2019correlated,landau2021macroscopic}. See Appendix~\ref{SuppComparison} for a more in-depth comparison between weakly and strongly low-rank networks.}

Strongly low-rank networks can be decomposed as~\cite{thibeault2024low}
\[
W=W_0+W_1
\]
where $W_0$ has rank $r\sim\mathcal O(1)$ and large singular values, while the singular values of $W_1$ are $\mathcal O(1)$ at most and $W_1$ can be full rank. We specifically assume that $W_1$ is a random matrix and independent from $W_0$. 
Since $W_0$ has rank $r$, we can write its singular value decomposition as
\[
W_0=U\Sigma V^T
\]
where $\Sigma$ is a diagonal, $r\times r$ matrix of singular values 
\[
\sigma_k=\Sigma_{kk}\gg 1
\] 
while $U=[\vvec u_1\cdots \vvec u_r]$ and $V=[\vvec v_1\cdots \vvec v_r]$ are $N\times r$ orthonormal matrices with columns $\vvec u_k$ and $\vvec v_k$ that define the left  and right singular vectors of $W_0$.



\subsection{Low-rank suppression occurs under Assumptions 1 and 2.}\label{SLRS}

Our first result is that Assumptions 1 and 2 are sufficient for low-rank suppression (as observed in Figures~\ref{F1}d and \ref{F2}d). 
Specifically, we claim that low-rank suppression is realized under a static input of the form
\begin{equation}\label{Exa}
\vvec x_a=\vvec u_k+\frac{W_1\vvec v_k-\vvec v_k}{\sigma_k}.
\end{equation}
where $k\le r$ so that $\sigma_k\gg 1$ is one of the dominant singular values. 
We claim that inputs of this form are strongly suppressed by the network. To see why this is true, first note that $\|\vvec x_a\|=1+o(1)$ since $\sigma_k\gg 1$, $\|\vvec u_k\|=\|\vvec v_k\|=1$, and the maximum singular value of  $W_1$ is $\mathcal O(1)$. 
Now note that the steady-state response to the input $\vvec x_a$ is given by 
\[
\vvec z_a=-\frac{\vvec v_k}{\sigma_k}.
\]
This can be checked by substituting the equations for $\vvec x_a$ and $\vvec z_a$ into the steady-state equation, $\vvec z=W\vvec z+\vvec x$. Since $\|\vvec v_k\|=1$ and $\sigma_k\gg 1$, we have that $\|\vvec z_a\|\ll 1$. Therefore, the network strongly suppresses the input $\vvec x_a$ from Eq.~\eqref{Exa} in the sense that $\|\vvec z_a\|\ll \|\vvec x_a\|$.

Note from Eq.~\eqref{Exa} that $\vvec x_a$ is closely aligned to $\vvec u_k$ since  $\sigma_k\gg 1$. Hence, the suppressed input, $\vvec x_a$, is nearly (but not exactly) parallel to the left singular vector, $\vvec u_k$. Also note that the response, $\vvec z_a$, is perfectly aligned to $\vvec v_k$. {Hence, strongly low-rank recurrent networks suppress inputs aligned near the dominant left singular vectors ($\vvec u_k$) and the suppressed response is aligned with the dominant right singular vectors ($\vvec v_k$).} 


Interestingly, these input-response directions along which suppression occurs are reversed from the input-response directions along which amplification occurs in a feedforward network with the same connectivity (for example, the network in Figure~\ref{SuppFigFfwd}  in Appendix~\ref{AppSuppFigures} for which $\vvec z=W\vvec x$ in the steady-state). More specifically, in a feedforward network with connectivity $W$, inputs aligned to the dominant \textit{right} singular vectors ($\vvec v_k$) produce \textit{amplified} responses in the direction of the corresponding \textit{left} singular vectors ($\vvec u_k$). This is a left right reversal and a suppression-amplification reversal of the behavior for recurrent networks, for which inputs aligned to $\vvec u_k$ produce suppressed responses in the direction of $\vvec v_k$.


In summary, recurrent networks with strongly low-rank structure satisfying Assumptions 1 and 2 admit a small number of input directions along which inputs are strongly suppressed. 
We call this form of suppression low-rank suppression. Low rank suppression occurs regardless of whether connectivity is normal, explaining why low-rank suppression was observed in both Figure~\ref{F1} and Figure~\ref{F2}. 
Later, we will explore consequences of low-rank suppression in common network structures. First, we discuss conditions for the  separate but related phenomenon of high-dimensional dynamics.


%


\subsection{Assumption 3: External inputs are stationary, high-dimensional, and slowly varying.}

{To begin discussing the dimensionality of $\vvec z(t)$, we need to make assumptions that allow us to define precisely what we mean by high- and low-dimensional dynamics.} To this end, we assume that $\vvec x(t)$ is a stationary ergodic stochastic process. In this case $\vvec z(t)$ is also a stationary ergodic process whenever dynamics are stable~\cite{gardiner1985handbook} as in Assumption~1. 
This allows us to quantify the dimensionality of $\vvec z(t)$ in terms of its principal components or, equivalently, by the singular values of its stationary covariance matrix. {Specifically, the variance explained by the first $k$ principal components represents the amount of variability constrained to a $k$-dimensional linear manifold. For this reason, we define $\vvec z(t)$ to have low-dimensional dynamics when there exists an $r\sim\mathcal O(1)$ such that $\sigma_j/\sigma_k\to 0$ as $N\to\infty$ whenever $k\le r$ and $j>r$. This implies that the vast majority of variability lies on a low-dimensional linear manifold. Otherwise, we say that $\vvec z(t)$ has high-dimensional dynamics. }

We  assume that $\vvec x(t)$ is high-dimensional. 
If $\vvec x(t)$ were low-dimensional, then low-dimensional dynamics in $\vvec z(t)$ could be inherited from $\vvec x(t)$ instead of being generated intrinsically  (see Figure~\ref{SuppFigLowDimInput}  in Appendix~\ref{AppSuppFigures} for an example). In assuming that $\vvec x(t)$ is high-dimensional, we are focusing on the question of whether low-dimensional dynamics arise from network interactions, not from low-dimensional input alone.

{For simplicity, within this section, we specifically assume that the components, $\vvec x_j(t)$, are i.i.d. This implies that $\vvec x(t)$ is high-dimensional, is a stronger assumption than necessary, but simplifies the exposition. In Appendix~\ref{SStochasticAnalysis}, we consider more general inputs that are not assumed to be i.i.d.}

For now, we also assume that $\vvec x(t)$ varies in time more slowly than the intrinsic timescale, $\tau$, of the network interactions. To make this assumption more precise, note that whenever $\vvec x(t)=\vvec x$ is constant, solutions, $\vvec z(t)$, converge to the steady-state given by $\vvec z=[I-W]^{-1}\vvec x$. We assume that $\vvec x(t)$ varies sufficiently slowly that $\vvec z(t)$ is approximated by the quasi-steady state approximation,
\begin{equation}\label{Equasi}
\vvec z(t)\approx [I-W]^{-1}\vvec x(t).
\end{equation}
In Appendix~\ref{SStochasticAnalysis}, we give more details about this assumption. {In Appendix~\ref{Sfast}, we extend our analysis and results to the case that $\vvec x(t)$ is not slowly varying.}

\subsection{Conditions for low-dimensional dynamics under Assumptions 1--3}\label{SLDD}

 \begin{figure*}
 \centering{
 \includegraphics[width=5in]{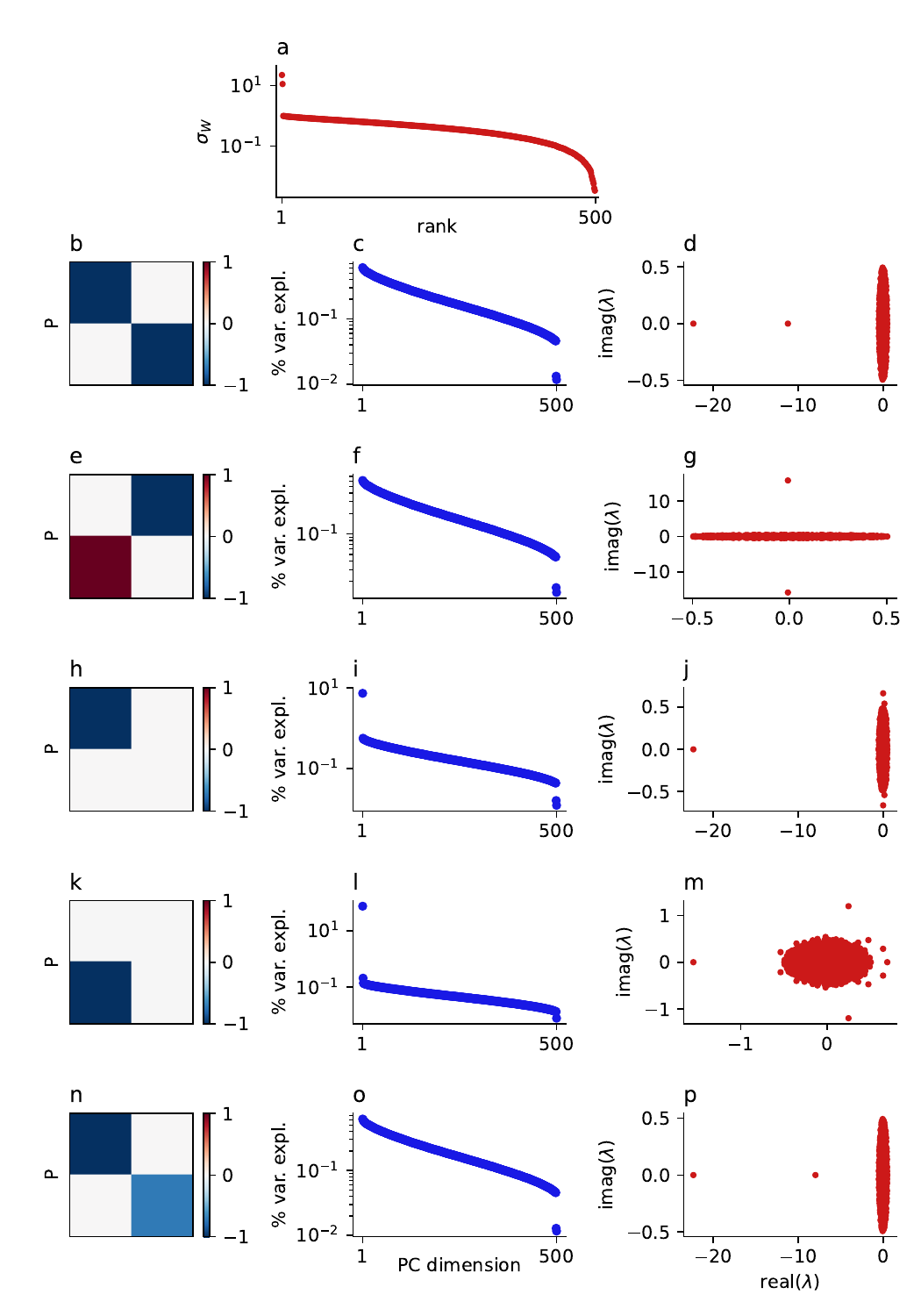}
 }
 \caption{{\bf Conditions for high-dimensional dynamics in a network with rank-two structure.} {\bf a)} Singular values of $W$ demonstrate an effective low-rank structure. {\bf b)} The  alignment matrix, $P$, when $W_0$ is normal and symmetric. {\bf c)} The variance explained by each principal component of the network dynamics, $\vvec z(t)$. {\bf d)} The eigenvalues of $W$. {\bf e--g)} Same as b--d except $W_0$ is EP, but non-normal.  {\bf h--j)} Same as b--d except $W_0$ is not EP. {\bf k--m)} Same as h--j but using a different non-EP structure. {\bf n--p)} Same as b--d except $W_0$ is non-EP with a non-vanishing EP component. See the text for the precise definition of $W_0$ in each case.}
 \label{F3}
 \end{figure*}

Under the quasi-steady state approximation in Eq.~\eqref{Equasi} and the assumption that elements of $\vvec x(t)$ are i.i.d., it is easy to show that the covariance matrix of $\vvec z(t)$ can be written as~\cite{gardiner1985handbook}
\begin{equation*}
\overline C^\vvec z \propto [I-W]^{-1}[I-W]^{-T}.
\end{equation*}
As a result, the question of whether $\vvec z(t)$ is low-dimensional becomes a question of whether $[I-W]^{-1}$ has a small number of dominant singular values whenever $W$ has a small number of dominant singular values. More details are given in Appendix~\ref{SStochasticAnalysis} where we prove that the answer to this question depends on the recurrent alignment matrix~\cite{landau2021macroscopic} of $W_0$, defined by
\[
P=V^TU.
\]
This $r\times r$ matrix measures the alignment between the left and right singular vectors of $W_0$. Specifically, $P_{jk}=\vvec v_j\cdot \vvec u_k$ measures the alignment between the $j$th right singular vector and the $k$th left singular vector of $W_0$. This matrix is also sometimes referred to as the ``overlap matrix''~\cite{clark2025connectivity}. 
Our main result from Appendix~\ref{SStochasticAnalysis} can be stated as follows:
\begin{claim}
 Consider the dynamics defined by Eq.~\eqref{EzDS} under Assumptions 1--3. If the dynamics of $\vvec z(t)$ are low-dimensional then $P=V^TU$ has at least one asymptotically small singular value. 
\end{claim}
Note that this result is equivalent to the following:
\setcounter{claim}{0}
\begin{claim}
 Consider the dynamics defined by Eq.~\eqref{EzDS} under Assumptions 1--3. If all singular values of $P=V^TU$ are $\mathcal O(1)$ then the dynamics of $\vvec z(t)$ are high-dimensional.
\end{claim}
Recall that all singular values of $P$ are bounded by $1$, so the condition that all singular values of $P$ are $\mathcal O(1)$ simply says that no singular values of $P$ are asymptotically small, hence the two versions of Claim 1 are logical contrapositives. {We were unable to prove the converse of these claims (that dynamics are low-dimensional whenever $P$ has a small singular value), even though it is consistent with all of our simulations. Proving the converse or finding a counterexample could be a fruitful direction for future research.}

{
Our proof of Claim~1  relied on an assumption that $\vvec x(t)$ varies more slowly than the timescale, $\tau$, of network interactions. In Appendix~\ref{Sfast}, we consider the case where $\vvec x(t)$ is not necessarily slow. In that case, the dimensionality can be broken down into the contribution from different frequency modes. When $\vvec x(t)$ varies at both high and low frequencies, the results in Claim~1 are still valid (see Figure~\ref{SuppFigFastx} for an example).}

To better understand the conditions on $P$ in Claim~1, we consider simulations of a widely used model from computational neuroscience~\cite{sompolinsky1988chaos,sussillo2009generating,mastrogiuseppe2018linking,rosenbaum2024modeling},
\begin{equation}\label{EdzdtNonLin}
\tau\frac{d\vvec z}{dt}=-\vvec z+W\tanh(\vvec z)+\vvec x
\end{equation}
with $\tau=1$. 
In these simulations, we take each $\vvec x_j(t)$ to be a smooth, i.i.d.~Gaussian process with correlation time $\tau_x=10$. 
We take $W_0$ to have rank $r=2$ so $W$ has two dominant singular values (Figure~\ref{F3}a) and we fix 
\[
U=\left[\begin{array}{cc}\vvec u_1 & \vvec u_2\end{array}\right]
\]
while exploring different choices of $V$.

We first consider a case in which $W_0$ is normal and symmetric by setting 
\[
\begin{aligned}
V&=\left[\begin{array}{cc}-\vvec u_1 & -\vvec u_2\end{array}\right].
\end{aligned}
\]
This is an extension of the example from Figure~\ref{F1} to rank $r=2$. In this case, $P$ is diagonal with $P_{kk}=-1$ (Figure~\ref{F3}b) so all singular values of $P$ are $\sigma_P=1$ and the network produces  high-dimensional dynamics (Figure~\ref{F3}c). Like the example from Figure~\ref{F1}, $W_0$ has eigenvalues with large, negative real part (Figure~\ref{F3}d). 
In this example, $W_0$ is both normal and symmetric. Results are similar when $W_0$ is normal and asymmetric (obtained, for example, by applying a rotation to $V$ in $r=2$ dimensions) except the conclusion $|P_{kk}|=1$ only holds when the non-zero eigenvalues of $W_0$ are real, which is necessarily the case when the $\sigma_k$ are distinct.

High-dimensional dynamics do not require that $W_0$ is normal. Instead, it is sufficient that $U$ and $V$ share a column space, $\textrm{col}(U)=\textrm{col}(V)$. Low-rank matrices, $W_0=U\Sigma V^T$, for which $U$ and $V$ share a column space are called EP matrices~\cite{schwerdtfeger1950introduction}. The term ``EP'' is sometimes said to refer to ``equal projectors'' 

All normal matrices are EP, but an EP matrix is not necessarily normal. If $W_0$ is an EP matrix, then $\sigma_P=1$ for all singular values of $P$, so the condition in Claim~1 is satisfied. 

An example of a non-normal, EP matrix, $W_0$, is given by taking
\[
\begin{aligned}
V&=\left[\begin{array}{cc}\vvec u_2 & -\vvec u_1\end{array}\right].
\end{aligned}
\]
The resulting matrix $W_0=U\Sigma V^T$ is highly non-normal because $\vvec u_1\cdot \vvec v_1=\vvec u_2\cdot \vvec v_2=0$. For this example, $P$ is zero along the diagonal (Figure~\ref{F3}e), but  all singular values of $P$ are $\sigma_P=1$, so the network still produces high-dimensional dynamics (Figure~\ref{F3}f). 
 Note that $W$ does not have any eigenvalues with large, negative real part (Figure~\ref{F3}g). Hence, this example combined with the previous example (Figure~\ref{F3}b--d) show that high-dimensional dynamics can arise in the presence or absence of eigenvalues with large, negative real part.

To obtain an example of a network that produces \textit{low}-dimensional dynamics (see also the example in Figure~\ref{F2}), we  set 
\[
\begin{aligned}
V&=\left[\begin{array}{cc}-\vvec u_1 & \vvec v_\perp \end{array}\right].
\end{aligned}
\]
where $\vvec v_\perp$ is orthogonal to $\vvec u_1$ and $\vvec u_2$, so  
$P$ has one singular value at $\sigma_P=1$ and another at $\sigma_P=0$ (Figure~\ref{F3}h). Because of the singular value at zero, the condition in Claim 1 is not met, and the network produces low-dimensional dynamics in which one principal component captures an outsized proportion of the variability (Figure~\ref{F3}i). For this example, $W$ has an eigenvalue with large, negative real part (Figure~\ref{F3}j). However, low-dimensional dynamics also arise in the case (Figure~\ref{F3}k,l)
\[
\begin{aligned}
V&=\left[\begin{array}{cc}-\vvec u_2 & \vvec v_\perp \end{array}\right].
\end{aligned}
\]
in which case $W$ lacks eigenvalues with large, negative real part (Figure~\ref{F3}m). Hence, these two examples demonstrate that low-dimensional dynamics can arise in the presence or absence of eigenvalues with large, negative real part.

So far, in all of the examples with high-dimensional dynamics, $W_0$ is an EP matrix. However, $W_0$ need not be fully EP for dynamics to be high-dimensional. 
We now consider the case in which left and right singular vectors are not perfectly aligned, but have some non-vanishing overlap, 
\begin{equation*}
\begin{aligned}
V&=\sqrt{1-c}\left[\begin{array}{cc}-\vvec u_1 &\,\vvec v_\perp\end{array}\right]+\sqrt{c}\left[\begin{array}{cc}\vvec -\vvec u_1 &\,-\vvec u_2\end{array}\right]
\end{aligned}
\end{equation*}
where $0<c<1$. This is a linear combination of the EP and non-EP matrices from the examples above. 
In this case, $W_0$ is not EP, but $P$ has a singular value at $\sigma_P=\sqrt c$ and the other at $\sigma_P=1$ (Figure~\ref{F3}n). 
As long as $c$ is not close to zero, the condition in Claim~1 is satisfied, and the network produces high-dimensional dynamics  (Figure~\ref{F3}o). Also, $W$ has eigenvalues with large, negative real part (Figure~\ref{F3}p). From this example, we may conclude that the low-rank part of the connectivity matrix only needs to have a non-vanishing EP component to produce high-dimensional dynamics. 


The numerical results in Figure~\ref{F3} provide an interpretation of Claim 1: If $W_0$ has a non-vanishing component with the EP property, then dynamics are high-dimensional. To our knowledge, this connection between the EP property of matrices and the dimensionality of network dynamics has not been made before.

The examples in Figure~\ref{F3} also demonstrate that the eigenvalue spectrum is generally not useful for determining whether dynamics are high- or low-dimensional in strongly low-rank networks near a stable equilibrium with high-dimensional input. {Instead, as Claim~1 shows, the singular values of the recurrent alignment matrix, $P$, are key to determining the dimensionality of dynamics in these networks.}

{
When low-dimensional dynamics are present, they arise from an effectively feedforward mechanism that takes variability from $\vvec x(t)$ and $W_1\vvec z(t)$ aligned with the column space of $V$ and maps it to strong variability along the column space of $U$ (see Appendix~\ref{SStochasticAnalysis} for more details). This mechanism is a form of non-normal amplification~\cite{trefethen1991pseudospectra,murphy2009balanced,hennequin2012non}. Our results add some new results to the literature on non-normal amplification under the conditions in Assumptions 1--3. First, our results show that this form of non-normal amplification requires that $W$ is highly non-EP, not just highly non-normal. See, for example, the example in Figure~\ref{F3}e--g, in which $W_0$ is highly non-normal, but EP so that amplification does not occur. Secondly, our results on low-rank suppression show that the amplified responses in the direction of $U$ associated with non-normal amplification are accompanied by suppressed responses approximately aligned to $V$ (see Figure~\ref{F2}e,f). Finally, Claim~1 shows that non-normal amplification is the \textit{only} way to generate low-dimensional dynamics under Assumptions 1--3.
}

It is common in theoretical work to consider models in which  the entries of $U$ and $V$ are random and independent with zero mean. In this case, $\vvec u_j$ and $\vvec v_k$ are nearly orthogonal so $\sigma_P\ll 1$ and the network produces low-dimensional dynamics (Figure~\ref{SuppFigNonNormal} in Appendix~\ref{AppSuppFigures}; compare to Figure~\ref{F2}). However, many networks arising in nature do not have purely random  structure. Next, we show  that many naturally arising network structures satisfy our conditions for  high-dimensional dynamics. 

\section{Biased weights and modular networks}

 \begin{figure*}
 \centering{
 \includegraphics[width=7in]{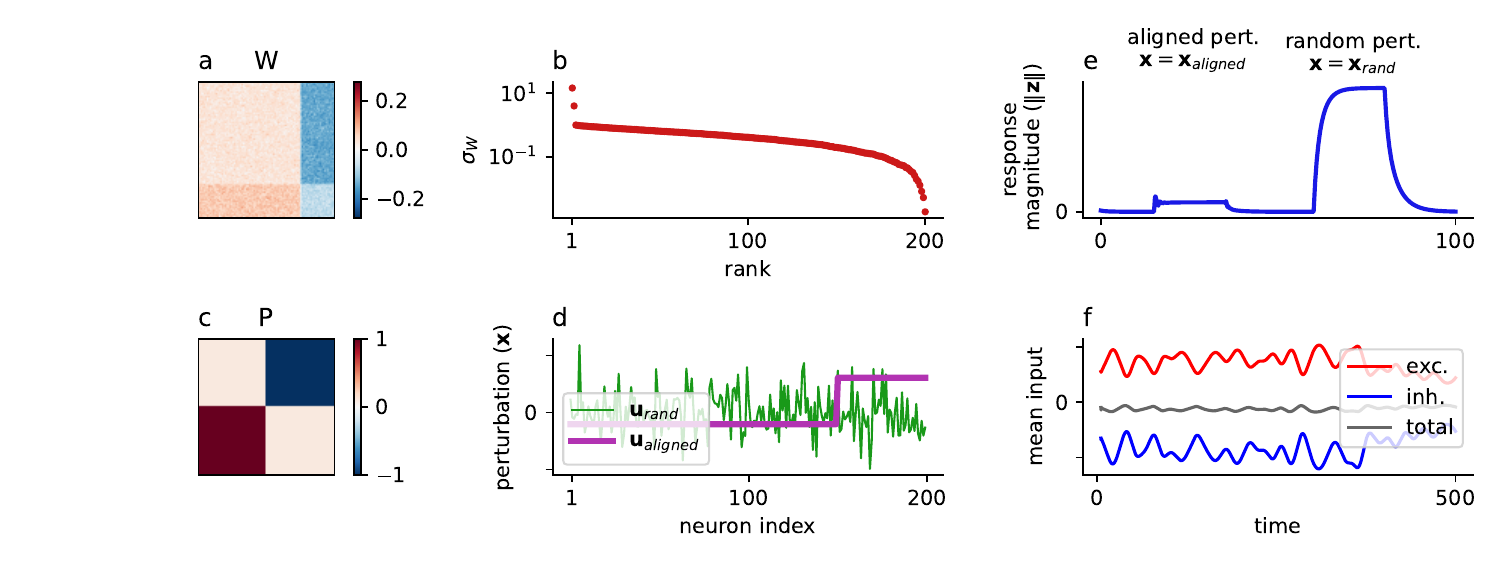}
 }
 \caption{{\bf Low-rank suppression and excitatory-inhibitory balance in a modular network.} {\bf a,b)} A modular network with biased blocks modeling excitatory and inhibitory neurons has low-rank structure. {\bf c)} The alignment matrix shows that the network is EP, but not normal. {\bf d)} An input that is constant within each block (purple) is aligned to the low-rank part, but a random input (green) is not. {\bf e)} {The norm of $\vvec z$ is more than 13 times smaller in response to the aligned versus random perturbation.} 
 {\bf f)} The excitatory (positive; red) component of mean input (local and external combined) balances with the inhibitory (negative; blue) component to produce a much smaller total input (gray), a widely observed phenomenon in neural circuits. 
 Network dynamics obey Eq.~\eqref{EdzdtNonLin}. 
 }
 \label{F4}
 \end{figure*}

So far, we considered networks with unbiased weights, $E[W_{jk}]=0$, which is common in modeling studies, but many networks in nature have weights with non-zero mean. Biased weights can produce  low-rank structure. As a simple example, consider a random network with independent weights satisfying
\[
E[W_{jk}]=m\;\textrm{ and }\; \textrm{std}(W_{jk})=s.
\]
If $m$ and $s$ scale similarly, then the largest singular value of $W$ is near $\sigma_1=|m|N$ while the next-largest singular value  scales like $\sigma_2= \mathcal O(s \sqrt N)$,  implying  an effective rank-one structure when $N$ is large. 
The dominant rank-one part has constant entries, so it is normal and the network exhibits low-rank suppression and high-dimensional dynamics (Figure~\ref{SuppFigBias}  in Appendix~\ref{AppSuppFigures}; see Appendix B in~\cite{mackay1990analysis} for related results).

More generally, modular structure arises in many natural settings~\cite{thibeault2024low,gu2024emergence}. Specifically, many networks in nature represent the interaction between $n$ populations, and mean connection weights between populations are often non-zero. The adjacency matrices of modular networks can be arranged to have a block structure 
\[
W=\left[\begin{array}{ccc}W^{1,1} &  \cdots & W^{1,n} \\ W^{2,1} &  \cdots & W^{2,n} \\ \vdots & \cdots & \vdots\\ W^{n,1} & \cdots & W^{n,n}\end{array}\right]
\]
where $W^{a,b}$ is a $N_a\times N_b$ sub-matrix quantifying connections from population $b$ to population $a$. In general, each sub-matrix can have a different, non-zero mean and standard deviation,
\[
E\left[W^{a,b}_{jk}\right]=m_{ab}\;\textrm{ and }\; \textrm{std}\left(W^{a,b}_{jk}\right)=s_{ab}.
\]
If we assume that  each population has $N_a\sim\mathcal O(N)$ members and that the $m_{ab}$ scale similarly to the $s_{ab}$, then $W$ has up to $n$ dominant singular values that scale like $\mathcal O(m_{ab}N)$ and the remaining singular values scale like $\mathcal O(s_{ab}\sqrt N)$. 
Hence, when $s_{ab}\sim m_{ab}$, large modular networks with biased weights naturally produce a low-rank structure in the sense that a small number of singular values are asymptotically larger than the rest~\cite{thibeault2024low}.  If additionally $m_{ab}\gg \mathcal O(1/N)$, then the dominant singular values are asymptotically large, so the network is strongly low-rank in the sense defined in  Assumption 2.

Networks of this form can be decomposed as $W=W_0+W_1$ where 
$
W_0=E[W]
$
is constant within each block, so $W_0$ has rank at most $n$. In general, $W_0$ is not a normal matrix, but it is an EP matrix because $\textrm{col}(U)$ and $\textrm{col}(V)$ are each spanned by the indicator vectors of the $n$ populations,
\[
\textrm{col}(U)=\textrm{col}(V)=\textrm{span}\left\{\vvec 1_1,\vvec 1_2,\ldots,\vvec 1_n\right\}
\]
where each $\vvec 1_k$ is an $N$-dimensional indicator vector with entries defined by 
\[
[\vvec 1_k]_{j} =\begin{cases}1 & \textrm{if index $j$ is in population $k$}\\ 0 &\textrm{otherwise}\end{cases}
\]
Therefore, modular networks with biased weights exhibit low-rank suppression and high-dimensional dynamics. 

As a specific example of a modular network, we consider a model of a local neuronal network in the cerebral cortex. Cortical neurons obey Dale's Law: All outgoing connections from a particular neuron have the same polarity, positive for excitatory neurons and negative for inhibitory neurons, and mean connection weights also depend on the postsynaptic neuron type~\cite{levy2012spatial,pfeffer2013inhibition}. These properties produce a modular structure  in which the columns of the adjacency matrix corresponding to excitatory neurons are non-negative, while the columns corresponding to inhibitory neurons are non-positive. Without loss of generality, we can order the neurons so that the first $N_e$ neurons are excitatory and the remaining $N_i$ are inhibitory (with $N=N_e+N_i$). In this case, $W$ has a $2\times 2$ block structure (Figure~\ref{F4}a).
We also assume that $N_e,N_i\sim\mathcal O(N)$, consistent with the fact that around 80$\%$ of neurons in cortex are excitatory.  We consider the case in which $m_{ab}$ and $s_{ab}$ scale like $1/\sqrt N$, consistent with experiments~\cite{barral2016synaptic} and theoretical work~\cite{landau2021macroscopic,van1996chaos,renart2010asynchronous,rosenbaum2017spatial,baker2019correlated,o2022direct}. 

Together, these biologically justified assumptions imply that the network has a strongly low rank structure, as defined in Assumption 2, with rank $r=2$ (Figure~\ref{F4}b). Specifically, two singular values scale like $\sqrt N$ while the others are $\mathcal O(1)$. Moreover, the low-rank part of the connectivity matrix is EP (Figure~\ref{F4}c). Hence, the network produces high-dimensional dynamics and low-rank suppression despite the fact that $W$ is low-dimensional and non-normal.  
The column space of $U$ and $V$ consist of all vectors that are uniform within each population,
\[
\begin{aligned}
\textrm{col}(U)=\textrm{col}(V)&=\textrm{span}\{\vvec 1_e,\vvec 1_i\}\\
&=\{[a\;\cdots\;a\; b\;\cdots\; b]^T\,|\, a,b\in\R\}
\end{aligned}
\]
where $\vvec 1_e$ and $\vvec 1_i$ are indicator vectors for the excitatory and inhibitory populations. 
Therefore, external inputs, $\vvec x$, that are uniform within each population are aligned to the low-rank part and suppressed relative to random perturbations (Figure~\ref{F4}d,e). In other words, external inputs that stimulate all excitatory neurons equally and all inhibitory neurons equally are suppressed relative to inputs that are  inhomogeneous within one or both populations, consistent with previous work on balanced network models~\cite{o2022direct,pyle2016highly,landau2016impact,ebsch2018imbalanced}. 
As predicted, simulations also show high-dimensional responses to high-dimensional stimuli (Figure~\ref{SuppFigModular} in Appendix~\ref{AppSuppFigures}), consistent with observations that neural responses in monkey visual cortex are high-dimensional when visual stimuli are high-dimensional~\cite{stringer2019high}.

Averaging over the excitatory and inhibitory populations in this network represents a projection onto $\textrm{col}(U)=\textrm{col}(V)$. As a result, the cancellation mechanism illustrated in Figure~\ref{F1}h manifests as a tight balance between mean excitatory (positive) and inhibitory (negative) input to neurons (Figure~\ref{F4}f), a phenomenon that is widely observed in neural recordings~\cite{okun2008instantaneous,atallah2009instantaneous,adesnik2010lateral,Shu:2003ht,wehr:2003,Haider:2006gs,Dorrn:2010hu} and widely studied in computational models~\cite{van1996chaos,renart2010asynchronous,ebsch2018imbalanced,landau2021macroscopic}. Hence, the widely studied theory of balanced networks can be interpreted as a special case of the low-rank suppression and low-rank cancellation studied here, specific to the particular modular network structure arising from Dale's law.

 \begin{figure*}
 \centering{
 \includegraphics[width=5.5in]{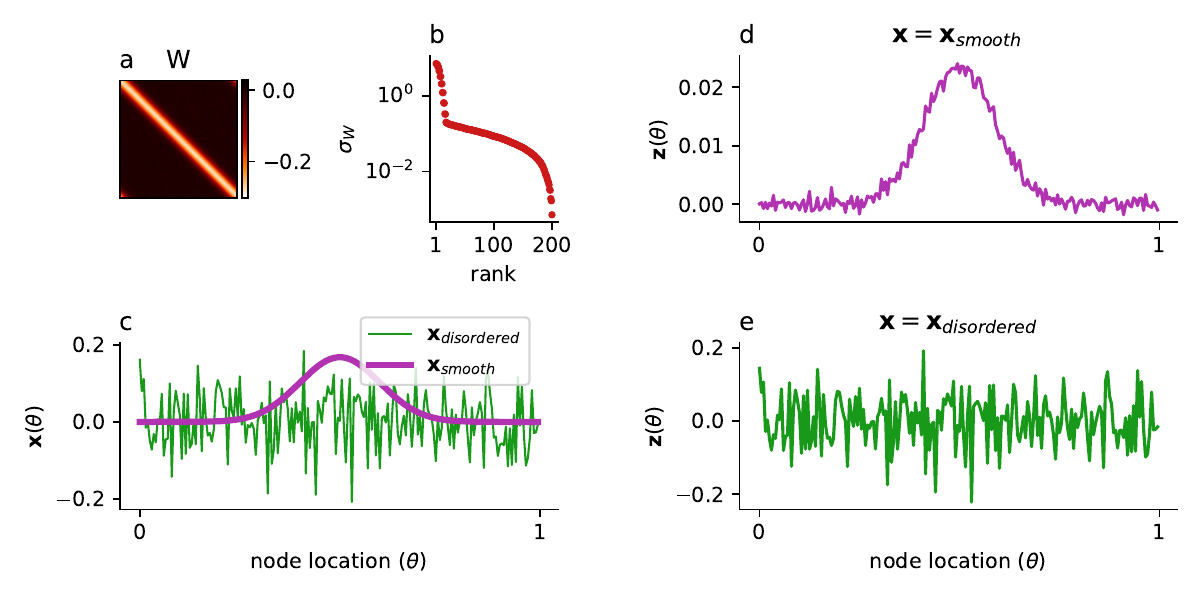}
 }
 \caption{{\bf Low-rank suppression in a network with spatial structure.} {\bf a)} Connectivity matrix, $W$. Connection strength is a Gaussian function of distance. {\bf b)} Singular values of $W$ demonstrate effective low-rank structure. {\bf c)} A spatially smooth perturbation (purple) is aligned to the low-rank part of $W$ while a spatially disordered perturbation (green) is not. {\bf d,e)} Therefore, the network response to a smooth perturbation is much weaker than the response to a disordered perturbation (compare vertical tick marks). {The norm of $\vvec z$ is more than 7 times smaller in response to the smooth versus disordered perturbation.} 
 {Network dynamics obeyed Eq.~\eqref{EdzdtNonLin}}.}
 \label{F5}
 \end{figure*}

\section{Amplification of disordered perturbations in networks with spatial structure}

Many networks in nature exhibit connection strength that depends smoothly on the distance between nodes in physical or other spaces, resulting in an effective low-rank structure~\cite{thibeault2024low,rosenbaum2017spatial}. As a simple example, we consider a model in which each node is assigned a spatial location, $\theta\in[0,1)$, and connection weights in $W_0$ decay like a Guassian function of the distance, $d\theta$, between nodes computed with periodic boundaries. We assume strong coupling in the sense that individual weights scale like $1/\sqrt N$. In this case, only nearby nodes are  strongly connected (Figure~\ref{F5}a).  Note that stability requires that connections in $W_0$ are negative since coupling is strong.  Connectivity  is also perturbed by a random component, $W_1$, as above. This connectivity structure is effectively low-rank (Figure~\ref{F5}b) and the low-rank part is normal, with left  and right singular vectors spanned by the low-frequency Fourier basis vectors (Figure~\ref{SuppFigFourier} in Appendix~\ref{AppSuppFigures}). Therefore, networks of this form satisfy our conditions for low-rank suppression and high-dimensional dynamics (Figure~\ref{SuppFigSpace} in Appendix~\ref{AppSuppFigures}).

Perturbations that are smooth in space are aligned to the low-rank part of the connectivity matrix because they are formed by sums of low-frequency spatial Fourier modes. Hence, perhaps surprisingly, strongly connected networks with spatial structure are more sensitive to spatially disordered perturbations than to spatially smooth perturbations (Figure~\ref{F5}c-e; compare tick labels on vertical axes).

\section{Low-rank suppression and high-dimensional dynamics in an epidemiological network.}

We next consider a real epidemiological network, specifically a network of high school social contacts~\cite{mastrandrea2015contact}, which was used in recent theoretical work on low-rank network dynamics~\cite{thibeault2024low}. In that work, the authors considered a quenched mean-field reduction of the susceptible-infected-susceptible  model, 
\begin{equation}\label{SISa}
\tau\frac{d\vvec z}{dt}=-\vvec z+\gamma(1-\vvec z)\circ Q\vvec z
\end{equation}
where $\gamma>0$ and $\circ$ denotes element-wise multiplication. Each $\vvec z_j(t)$ models the probability that an individual is infected. 
Here, $Q$ is the proportional to the proximity matrix of 637 high school students, indicating whether each pair of students were in proximity of each other during a specific school week~\cite{mastrandrea2015contact}. 
This matrix  is effectively low-rank in the sense that it has a small number of dominant singular values (Figure~\ref{SuppFigEpidemiological}a in Appendix~\ref{AppSuppFigures}). In~\cite{thibeault2024low}, it was shown that dynamics generated by Eq.~\eqref{SISa} on this network are effectively low-dimensional.

  \begin{figure}
 \centering{
 \includegraphics[width=2.75in]{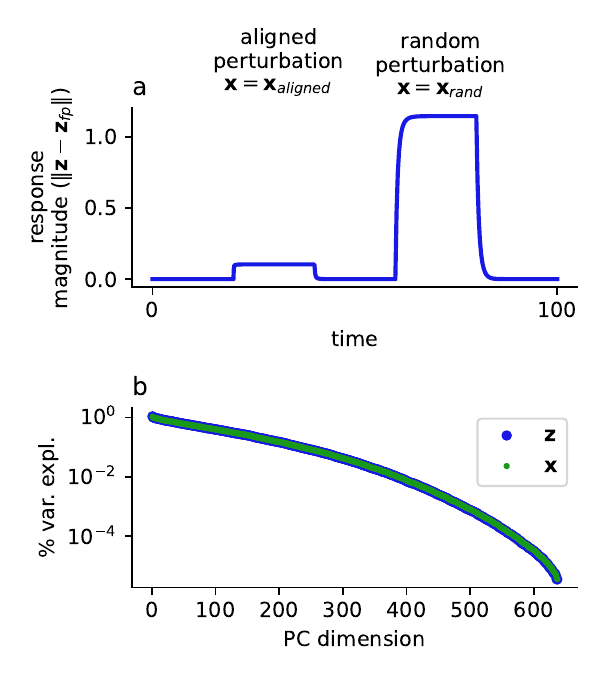}
 }
 \caption{{\bf Low-rank suppression and high-dimensional dynamics in a real epidemiological network.} {\bf a)} The response to a perturbation aligned with the dominant low-rank part is weaker than the response to a perturbation of the same magnitude in a random direction. {The distance of $\vvec z$ from its fixed point is more than 10 times smaller in response to the aligned versus random perturbation.}  {\bf b)} In the presence of high-dimensional random perturbations, the variance explained by each principal component of the dynamics is similar to the input, indicating high-dimensional dynamics. Network dynamics obeyed Eq.~\eqref{EdzdtNonLin}.}
 \label{F6}
 \end{figure}
  
However, Eq.~\eqref{SISa} is completely self-contained without any  perturbations. {Perturbations arise in epidemiological dynamics through interactions with individuals from outside of the modeled network and through the natural stochasticity of infection. Specifically, $\vvec z(t)$ represents a probability of infection. Actual infection events will be random, which can be approximated by adding a stochastic forcing term to Eq.~\eqref{SISa}. Moreover, students interact with siblings and parents at home (who are not included in the network represented by $Q$), which produces additional external forcing. 
As a simple model of these perturbations, we modified the model by adding an external forcing term to obtain}
\begin{equation}\label{SISx}
\tau\frac{d\vvec z}{dt}=-\vvec z+\gamma(1-\vvec z)\circ \left[Q\vvec z+\vvec x(t)\circ \vvec z\right].
\end{equation}

Eq.~\eqref{SISx} admits a fixed point when $\vvec x(t)=0$. Following the linearization discussed in Section~\ref{SLin} and Appendix~\ref{SuppLin}, the effective connectivity, $W$, is defined by the Jacobian matrix at the fixed point. We find that, like $Q$, the effective connectivity matrix, $W$, has a small number of dominant singular values, so our theory predicts that this network exhibits low-rank suppression. Moreover, the recurrent alignment matrix, $P$, defined from $W$ satisfies the conditions for high-dimensional dynamics from Claim~1 (Figure~\ref{SuppFigEpidemiological}a in Appendix~\ref{AppSuppFigures}). 
Simulations of Eq.~\eqref{SISx} confirm that the model exhibits low-rank suppression (Figure~\ref{F6}a and  Figure~\ref{SuppFigEpidemiological}b--d in Appendix~\ref{AppSuppFigures}) and high-dimensional dynamics (Figure~\ref{F6}b) in contrast to the model without external input studied in~\cite{thibeault2024low}.

Importantly, the ``aligned'' input direction that we used to demonstrate low-rank suppression in Figure~\ref{F6}a is defined in terms of the dominant left singular vector of the \textit{effective} connectivity matrix, $W$, \textit{i.e.}, the Jacobian at the fixed point. Using the left singular vector of the explicit connectivity, $Q$, in place of the effective connectivity, $W$, does not produce suppression (see Figure~\ref{SuppFigEpidemiological}c,d in Appendix~\ref{AppSuppFigures}). This observation validates our linearization argument and highlights the necessity of using effective connectivity in place of explicit connectivity in nonlinear networks.

Taken together, our results highlight the importance of accounting for external perturbations and internal noise when studying the dimensionality of epidemiological dynamics. Moreover, the results imply that epidemiological networks can be  more sensitive to random perturbations than to perturbations aligned with the network's low-rank structure. For example, if the forcing term is taken to represent interventions like masking or inoculation, then our results predict that it is more effective to apply these interventions randomly to individuals than to apply interventions to individual sub-populations or other structures reflected in the dominant structure of the proximity matrix.

\section{Discussion}

We presented theory and examples showing that strongly low-rank recurrent networks close to a stable equilibrium suppress inputs or perturbations aligned with the dominant directions of their connectivity matrices. We also derived conditions under which these networks generate high-dimensional dynamics.  We showed that many low-rank structures that arise in nature are consistent with low-rank suppression and high-dimensional dynamics.

Recent parallel work~\cite{mastrogiuseppe2025stochastic} also demonstrated that low-rank networks can produce suppression and high-dimensional dynamics under some assumptions. In that work, connectivity matrices had a small number of non-zero singular values which were $\mathcal O(1)$ in magnitude, in contrast to the large singular values assumed in our models (see Appendix~\ref{SuppComparison} for a more complete comparison). Other related work showed that trained neural networks can exploit suppression to cancel noise~\cite{schuessler2024aligned}.

{
Our results  rely on a linearization of the dynamics around a stable fixed point. The dimensionality of nonlinear network dynamics away from a stable equilibrium is also an important question. However, nonlinear networks can produce any dynamics in response to any inputs~\cite{maass2002real,jaeger2004harnessing} and effective connectivity can change with network state, so it is difficult to draw broad, general conclusions about the dimensionality of dynamics in nonlinear networks. Previous work~\cite{thibeault2024low} circumvented this difficulty by parameterizing the model in terms of the vector of local inputs from the network, $\vvec y=W\vvec z$, and then used bounds on the Jacobian matrix of the dynamics with respect to $\vvec y$, but they did not consider external inputs or perturbations. Combining our approaches by incorporating external inputs into their analytical approach is a promising direction for future work.}

{
We focused on the response of low-rank networks to external input, whereas much previous work focuses on the intrinsic dynamics of low-rank networks in the absence of external drive. Landau and Sompolinsky~\cite{landau2021macroscopic} consider a similar model to ours, but look at the interaction between externally driven responses and intrinsic, chaotic dynamics. However, they restrict to low-dimensional external inputs in the column space of the low-rank part of the connectivity matrix. Future work could combine these approaches to better understand the interaction between chaotic, intrinsic dynamics and high-dimensional external input.
}

It is tempting to ascribe low-rank suppression to the presence of an eigenvalue with large, negative real part (as in Figure~\ref{F1} and Figure~\ref{F3}b--d). A large, negative eigenvalue is well-known to produce a suppressive dynamic.  However, the results in Figure~\ref{F2} and Section~\ref{SLDD} demonstrate that the eigenvalue spectrum is not generally relevant to the presence or absence of low-rank suppression. 
{In other words, in recurrent networks driven by external input near a stable equilibrium, low-rank suppression depends only on the presence of large singular values of the connectivity matrix, $W$. 
High-dimensional dynamics, on the other hand, depend additionally on the singular values of the recurrent alignment matrix, $P=V^TU$.}

Our results  have implications for networks in nature and applications. For example, in neuroscience, the implications of low-rank recurrent connectivity on neural dynamics is a topic of intense research~\cite{mastrogiuseppe2018linking,landau2021macroscopic,dubreuil2022role,landau2018coherent,schuessler2020dynamics,valente2022extracting,cimevsa2023geometry,beiran2023parametric,mastrogiuseppe2025stochastic,clark2025connectivity}. Our results are consistent with an observation that networks of neurons produce high-dimensional activity in the context of high-dimensional stimuli or tasks, but low-dimensional dynamics in response to low-dimensional stimuli or tasks~\cite{stringer2019high}.

Beyond neuroscience, our results imply that perturbations to a low-rank network are more effective when they are delivered non-uniformly across sub-populations or space. More generally, perturbations to a low-rank network are more effective when they are not aligned to the network's low rank structure. This result can be leveraged to design and test more effective interventions, for example to epidemiological, ecological, or social networks.

\section{Methods}

All simulations and analyses were performed in Python using a combination of custom written PyTorch and NumPy code. All differential equations, except for Figure~\ref{F6} and  Figure~\ref{SuppFigEpidemiological}  in Appendix~\ref{AppSuppFigures}, were solved using a simple forward Euler scheme with a step size of $dt=0.01$. The simulations in Figure~\ref{F6} and  Figure \ref{SuppFigEpidemiological}  in Appendix~\ref{AppSuppFigures} were solved using a Runge-Kutta scheme adapted from the approach used in previous work~\cite{thibeault2024low}. {Networks in Figures~\ref{F1} and \ref{F2} used the linear model from Eq.~\eqref{Edzdt1}. Networks in Figures~\ref{F3}--\ref{F5} used the nonlinear model from Eq.~\eqref{EdzdtNonLin}. }
Code to produce all figures is available at\\ \texttt{https://github.com/RobertRosenbaum/HighDimLowDimCode}

\begin{acknowledgments}
We thank Ashok Litwin-Kumar and David G.~Clark for helpful conversations and comments on drafts of the manuscript. This work was supported by the Air Force Office of Scientific Research (AFOSR) awards numbers FA9550-21-1-0223 and FA9550-26-1-0004, and the National Science Foundation under a Neuronex award  number DBI-1707400.
\end{acknowledgments}

\appendix

\section{Comparison to  previous work on dynamics of low-rank networks}\label{SuppComparison}

{
Here, we focus on the response of networks to high-dimensional external input. 
Some previous work on low-rank networks and low-dimensional dynamics focuses on the dimensionality of \textit{intrinsic} dynamics generated by the network in the absence of external input or in the presence of low-dimensional external input. For example chaotic and other nonlinear dynamics in untrained and trained networks can be low-dimensional~(see \cite{sussillo2013opening,mante2013context,aljadeff2016low,huang2019circuit,thibeault2024low} and Figures~\ref{SuppFigUnstable} and \ref{SuppFigUnstable2} in Appendix~\ref{AppSuppFigures}).
In contrast, we are here primarily interested in the network's response to high-dimensional external input. In our model's, the network's activity is driven primarily by the external input (which is filtered by the network's internal structure) and not by intrinsic dynamics. When the network is strongly low-rank, we showed that low-dimensional dynamics can only arise when the recurrent alignment matrix has small singular values (see Claim~1). 
}

In this section we review the relationship between our models and the models from previous theoretical work on low-rank networks, especially work by Ostojic et al.~\cite{mastrogiuseppe2018linking,schuessler2020dynamics,dubreuil2022role,valente2022extracting,cimevsa2023geometry,beiran2023parametric}, work by Thibeault et al.~\cite{thibeault2024low}, and work by Landau and Sompolinsky~\cite{landau2018coherent,landau2021macroscopic}

All of these studies  consider recurrent networks with connectivity of the form $W=W_0+W_1$ where $W_0$ is low-rank and $W_1$ is a full rank random matrix, just like our model. We next describe the salient properties that distinguish these models from ours.

Thibeault et al.~\cite{thibeault2024low} assume that their networks are self-contained and do not receive any time varying external input. Specifically, they explicitly restrict their dynamics to models of the form
\[
\frac{d\vvec  z}{dt}=F(\vvec z,W\vvec z)
\]
which excludes the possibility of modeling a time-varying external input, like $\vvec x(t)$ in our model. Models of this form cannot describe networks that are part of a larger network, or networks that are modulated by time-varying, external factors. Similar assumptions were made in one study by Landau and Sompolinsky~\cite{landau2018coherent}. 
In other work by Landau and Sompolinsky~\cite{landau2021macroscopic}, external input was included, but this input was assumed to be perfectly aligned to the low-rank structure of the connectivity matrix and therefore low-dimensional.

The salient differences between our model and the models considered by Ostojic and colleagues~\cite{mastrogiuseppe2018linking,schuessler2020dynamics,dubreuil2022role,valente2022extracting,cimevsa2023geometry,beiran2023parametric} are more subtle. Like us, they  consider  external inputs that are not aligned to the low-rank part of $W$. Also, like us, they take $W=W_0+W_1$ where $W_1$ is full rank with random entries and the variance of the entries in $W_1$ scale like $\mathcal O(1/N)$ so that the maximum singular value of $W_1$ scales like $\mathcal O(1)$. 

However, in contrast to our models, the low-rank components of the networks considered by  Ostojic and colleagues take the form
\begin{equation}\label{EW0Ostojic}
W_0=\sum_{\mu=1}^r \frac{\vvec m_\mu\vvec n_\mu^T}{N}
\end{equation}
where $r$ is the rank and each $\vvec m_\mu$ and $\vvec n_\mu$ are $N\times 1$ vectors with entries that scale like $\mathcal O(1)$. Specifically, they are taken to be random vectors and the variance of their entries scales like $\mathcal O(1)$. Often, they are taken to be biased random vectors with a non-zero mean that also scales like $\mathcal O(1)$. Because of the $1/N$ factor in Eq.~\eqref{EW0Ostojic}, the variance of the entries in $W_0$ scale like $\mathcal O(1/N^2)$ and, when the entries are biased, the mean entry in $W_0$ scales like $\mathcal O(1/N)$. Regardless of whether entries are biased, the singular values of $W_0$ scale like $\mathcal O(1)$ in the models considered by Ostojic and colleagues~\cite{mastrogiuseppe2018linking}. Hence, the singular values of the low-rank part ($W_0$) and the random part ($W_1$) have the same scale whenever the spectral radius of $W_1$ is $\mathcal O(1)$, in contrast to our models in which the singular values of $W_0$ are considered to be asymptotically larger than the singular values of $W_1$. 
This difference in scaling is acknowledged by Ostojic et al.~\cite{mastrogiuseppe2018linking,schuessler2020dynamics,dubreuil2022role,valente2022extracting,cimevsa2023geometry,beiran2023parametric} who refer to their networks as ``weakly low-rank.'' 


{To summarize, when $W_1$ is a random matrix with $\mathcal O(1)$ spectral radius and when the entries of $W_0$ are biased, the primary difference between  many {strongly} and  {weakly} low-rank networks in practice is that the entries of $W_0$ scale like $\mathcal O(1/N)$ in weakly low rank networks while they are larger (typically $\mathcal O(1/\sqrt N)$) in many strongly low-rank networks, for example in work on balanced networks~\cite{landau2018coherent,landau2021macroscopic}.
The transition between these two regimes can be visualized by gradually scaling the magnitude of the low-rank part from weak to strong. Figure~\ref{SuppFigscalec} in Appendix~\ref{AppSuppFigures} shows that low-rank suppression arises gradually in this case. 
}

The distinction between strongly and weakly low-rank networks is more nuanced when the full-rank, random part, $W_1$, of the connectivity matrix is assumed to be weak or absent. For example, recent work by Mastrogiuseppe et al.~\cite{mastrogiuseppe2025stochastic} considered the case in which the full-rank part is absent ($W_1=0$) and the low-rank connectivity matrix, $W=W_0$, has $\mathcal O(1)$ singular values. While their use of $\mathcal O(1)$ singular values is similar to weakly low-rank networks from other work~\cite{mastrogiuseppe2018linking}, the absence of a random part means that their connectivity matrices are truly low-rank (not just approximately low-rank). In that work~\cite{mastrogiuseppe2025stochastic}, Mastrogiuseppe et al.~showed that the emergence of suppression and high-dimensional dynamics depends on the degree of overlap between the left  and right singular vectors in ways that are not captured by our asymptotic theory.

For the sake of comparison to the models in~\cite{mastrogiuseppe2025stochastic}, we consider a rank-one network structure like the one in Figure~\ref{F1}.
Specifically, we  take $W_1$ to be random with spectral radius $\rho>0$, and take 
\[
W_0=c\vvec u \vvec u^T.
\]
In Figure~\ref{F1} and its analysis, we assume that $\rho\sim\mathcal O(1)$. In this case, a strong (\textit{i.e.}, asymptotically dominating) low-rank structure requires that $|c|\gg 1$. Therefore, stability in this case requires that $c<0$, which produces strong, negative eigenvalues (as in Figure~\ref{F1}; however, consider also the examples in Figure~\ref{F2}, Figure~\ref{F3}k--m, and Figure~\ref{SuppFigNonNormal} which are strongly low-rank without large eigenvalues).  
We could alternatively take the random, full-rank part to be weak: $\rho\ll 1$ (or it could be absent, $\rho=0$, as in~\cite{mastrogiuseppe2025stochastic}). In this case, we can take the low-rank part to be moderate in magnitude ($|c|=\mathcal O(1)$) while maintaining an asymptotically dominant low-rank structure because $|c|\gg \rho$ when $|c|=\mathcal O(1)$ and $\rho\ll 1$. Repeating the simulation from Figure~\ref{F1} demonstrates a weakened form of low-rank suppression with high-dimensional dynamics when $c<0$ and an amplification of aligned inputs with weakly low-dimensional dynamics when $c>0$ (Figures~\ref{SuppFigWeak} and \ref{SuppFigWeak2} in Appendix~\ref{AppSuppFigures}). {We also repeated all of the simulations from Figure~\ref{F3} in networks with $\rho\ll 1$ and $\sigma_1,\sigma_2\sim\mathcal O(1)$, so that the network is still low-rank in the sense that the first two singular values dominate, but not strongly low-rank because they are $\mathcal O(1)$ (see Figure~\ref{SuppFigWeak3}a). Under these conditions, the dynamics were high-dimensional regardless of the structure of $W_0$ (Figure~\ref{SuppFigWeak3}b--f). To understand this result, note that low-dimensional dynamics in our models (Figures~\ref{F2} and \ref{F3}i,l) arise from strong amplification of a small number of directions by $W$. When the low-rank part of $W$ is weaker, these low-dimensional dynamics will be weaker or absent. In other words, when the dominant singular values of $W_0$ are $\mathcal O(1)$ and smaller, the identity in $[I-W]^{-1}$ from the quasi-stead state analysis more strongly dominates the $W$ (or, at least, the $W$ does not dominate the $I$). More directly, the $\vvec x(t)$ in Eq.~\eqref{Edzdt1} dominates the $W\vvec z$. Therefore, high-dimensional inputs are simply transferred to high-dimensional responses when $W$ is weaker.}

 When $c<0$, the network exhibits a weaker form of low-rank suppression along with high-dimensional responses (Figure~\ref{SuppFigWeak} in Appendix~\ref{AppSuppFigures}). Similar examples and results are considered in~\cite{mastrogiuseppe2025stochastic}. In conclusion, some of our theoretical results on low-rank suppression and high-dimensional dynamics require that the low-rank part is large in \textit{absolute} terms ($|c|\gg 1$), not just \textit{relative} terms ($|c|\gg \rho$). 

When $|c|=\mathcal O(1)$, we can  take $c>0$ while keeping the network stable with a dominant low-rank part (in contrast to strongly low rank networks like Figure~\ref{F1} where stability requires $c<0$). When $c>0$, inputs aligned to $\vvec u$ are  \textit{amplified} compared to random perturbations (Figure~\ref{SuppFigWeak2} in Appendix~\ref{AppSuppFigures}), reversing the trend of low-rank suppression. However, the amplification is weak. In response to high-dimensional input, the network dynamics have a dominant principal component direction, but the dominance is weaker than in other examples we have considered (Figure~\ref{SuppFigWeak2}d; compare to Figure~\ref{F2} and Figure~\ref{SuppFigNonNormal}) because stability imposes a bound on $c$ (specifically, $c<1$) and therefore on $\|W_0\|$.

Landau and Sompolinsky~\cite{landau2018coherent,landau2021macroscopic} consider strongly low-rank networks. However, as noted above, their external input, $\vvec x(t)$, is aligned to the low-rank part of $W$ and is therefore low-dimensional itself. 
Thibeault et al.~\cite{thibeault2024low} consider several network models, but their ``rank-perturbed Gaussian'' model is strongly low-rank and equivalent to the network structure we study in Figure~\ref{F1}. Complicating matters, Thibeault et al. directly compare their rank-perturbed Gaussian model to the networks in Ostojic et al., despite the fact that the scaling of their low-rank parts differ by a magnitude of $\sqrt N$. Specifically, in Thibeault et al.~\cite{thibeault2024low}, the low-rank part is defined by Eq.~\eqref{EW0Ostojic} where $\vvec m$ and $\vvec n$ are unbiased Gaussian random vectors. The entries of $\vvec m$ have $\mathcal O(1/N)$ variance while the entries of $\vvec n$ have $\mathcal O(1)$ variance, so the entries of $W_0$ have $\mathcal O(1/N)$ variance, in contrast to the $\mathcal O(1/N^2)$ variance used by Ostojic et al.~\cite{mastrogiuseppe2018linking}. Importantly, this means that the singular values of $W_0$ are $\mathcal O(\sqrt N)$ in the rank perturbed Gaussian model analyzed by Thibeault et al.~\cite{thibeault2024low}, but $\mathcal O(1)$ in the models by Ostojic et al. Hence, the rank perturbed Gaussian models considered by Thibeault et al. are strongly low-rank, in contrast to the weakly low-rank networks considered by Ostojic et al.

\section{Details for the linearization of dynamics around a stable equilibrium}\label{SuppLin}

Here, we make precise our assumption that dynamics represent a small perturbation around a stable equilibrium, and give a more precise derivation of the linearized dynamics.

We begin by considering nonlinear dynamics of the form
\[
\tau\frac{d\vvec z}{dt}=\vvec F(\vvec z,\boldsymbol \eta)
\] 
where $\boldsymbol \eta(t)$ is the external input before any change of coordinates is applied. 
We assume that there exists a constant-in-time, $\boldsymbol \eta(t)=\boldsymbol \eta_0$, that produces a stable steady-state solution, $\vvec z_0$. In other words, there exist $\boldsymbol \eta_0$ and $\vvec z_0$ such that 
\[
\vvec F(\vvec z_0,\boldsymbol \eta_0)=\vvec z_0.
\]
We additionally assume that this fixed point is hyperbolically stable. In other words, the Jacobian matrix 
\[
J=-I+\partial_{\vvec z}\vvec F(\vvec z_0,\boldsymbol \eta_0)
\]
has eigenvalues with strictly negative real part. Here, $I$ is the identity matrix and $\partial_{\vvec z}\vvec F(\vvec z_0,\boldsymbol \eta_0)$ is the Jacobian matrix of $\vvec F$ with respect to $\vvec z$ evaluated at the fixed point. We then consider a small perturbation around this fixed point driven by a perturbation to the forcing term,
\[
\boldsymbol \eta_p(t)=\boldsymbol \eta_0+\epsilon \boldsymbol \eta(t).
\]
The response, $\vvec z_p(t)$, of the network to the perturbed forcing term, $\boldsymbol \eta_p(t)$, can be written to linear order in $\epsilon$ as
\begin{equation*}
\vvec z_p(t)=\vvec z_0+\epsilon \vvec z(t)+\mathcal O(\epsilon^2).
\end{equation*}
The perturbation, $\vvec z(t)$, obeys the linearized equation
\begin{equation}\label{Edzdtgen}
\tau\frac{d\vvec z}{dt}=-\vvec z+W\vvec z+\vvec x(t)
\end{equation}
where
\begin{equation*}
W=\partial_{\vvec z}\vvec F(\vvec z_0,\boldsymbol \eta_0)
\end{equation*}
is interpreted as the effective connectivity matrix and
\[
\vvec x(t)=\partial_{\boldsymbol \eta}\vvec F(\vvec z_0,\boldsymbol \eta_0)\boldsymbol \eta(t)
\]
is interpreted as the effective input. Our assumption that the equilibrium is stable implies that all eigenvalues of $W$ have real part less than $1$. {We additionally assume that the real part of all eigenvalues of the Jacobian are bounded away from 0, meaning that the real part of all eigenvalues of $W$ are bounded below $1$. In other words, there is a $d>0$ with $d\sim\mathcal O(1)$ such that all eigenvalues of $W$ satisfy $\textrm{Re}(\lambda_W)<1-d$. In particular, this implies that $\rho<1-d$ where $\rho$ is the spectral radius of the random part, $W_1$, of $W$.} 

Note that Eq.~\eqref{Edzdtgen} is equivalent to the linear dynamics in Eq.~\eqref{Edzdt1}. Hence, the linear network dynamics in Eq.~\eqref{Edzdt1} are a sort of normal form for the linearized dynamics described here. Therefore, we are justified in focusing on the linear dynamics of Eq.~\eqref{Edzdt1} in the analysis in the Results. 

In many models, the perturbation is purely additive, $\vvec F(\vvec z,\boldsymbol \eta)=\vvec H(\vvec z)+\boldsymbol \eta$, so that the effective input is equal to the raw input, $\vvec x=\boldsymbol \eta$. In other models, the raw input can have a different structure and dimensionality to the effective input. For illustrative purposes, we consider a contrived example of such a model in Figure~\ref{SuppFigLowDimInput} in Appendix~\ref{AppSuppFigures}. 

If a model is fully linear, \textit{i.e.}, if a model is already written in the form of Eq.~\eqref{Edzdtgen}, then we do not need to assume that dynamics remain close to the equilibrium, so inputs and perturbations need not be weak when dynamics are linear. 

Some models are parameterized in terms of a connectivity matrix that is not equal to the Jacobian matrix. For example, the system might be written as $d\vvec z/dt=\vvec G(W\vvec z,\vvec x)$ where $W$ is meant to quantify a connectivity matrix, but $\partial_\vvec z \vvec G(W\vvec z_0,\vvec x_0)\ne W$. In this case, our interpretation of the effective connectivity matrix is not consistent with the original parameterization, but they will often share  structural properties such as their approximate dimensionality. We consider one such example from epidemiology  at the end of the Results. {In that example, we found that the use of effective connectivity in place of the explicit connectivity matrix is important for our results to hold (see Figure~\ref{SuppFigEpidemiological} in Appendix~\ref{AppSuppFigures})}

{
\section{Precise definitions of strongly low-rank networks and low-dimensional dynamics.}\label{SuppStrong}
As mentioned above, we focus our attention on strongly low-rank networks. Intuitively, this means that the connectivity matrix has a small number of large singular values and the rest are not large. To make this definition mathematically precise, we need to consider an asymptotic limit of large $N$, where $N$ is the network size. Let $\sigma_k$ be the $k$th largest singular value of $W$. In strongly low-rank networks, there exists an $r\sim\mathcal O(1)$ such that $\lim_{N\to\infty}\sigma_k=\infty$ for $k=1,\ldots, r$. In addition, $\sigma_k\le \mathcal O(1)$ for $k>r$. }

{
We additionally need to specify what we mean by ``low-dimensional'' and ``high-dimensional'' dynamics. Let $\lambda_k$ be the $k$th largest eigenvalue of the covariance matrix of $\vvec z(t)$, equivalently the $k$th largest singular value since covariance matrices are symmetric and positive definite. Note that $\lambda_k$ is a non-negative real number representing the amount of variance explained by the $k$th principal component of $\vvec z(t)$. Intuitively, $\vvec z(t)$ is low-dimensional whenever a large portion of the variance is explained by a small number of singular values. To make this precise, we say that the dynamics of $\vvec z(t)$ are low-dimensional whenever there is an $r\sim\mathcal O(1)$ such that $\lim_{N\to\infty}\lambda_j/\lambda_k=0$ whenever $j>r$ and $k\le r$. 
}

\section{Conditions for high-dimensional dynamics in response to stationary, stochastic perturbations.}\label{SStochasticAnalysis}

We now consider conditions for the emergence of high-dimensional dynamics. We assume that $\vvec x(t)$ is a stationary, ergodic stochastic process. Define the cross-spectral matrix, $\widetilde C^{\vvec x}(f)$, of $\vvec x(t)$ at frequency $f$ as the Fourier transform of the matrix of cross-covariance matrix,
\[
\widetilde C^{\vvec x}_{jk}(f)=\int_{-\infty}^\infty C_{jk}^{\vvec x}(\tau)e^{-2\pi i f \tau}d\tau
\]
where $C_{jk}^{\vvec x}(\tau)=\cov(\vvec x_j(t),\vvec x_k(t+\tau))$ is the stationary cross-covariance. 
When each $\vvec x_j(t)$ is i.i.d., then $\widetilde C^{\vvec x}(f)=I c_x(f)$ is a multiple of the identity matrix where $ c_x(f)$ is the power spectral density of each $\vvec x_j(t)$. The assumption of i.i.d. elements is true for all of the examples we consider, but we will not apply this simplification until later in our calculation.

The the cross-spectral density of the network response, $\vvec z(t)$, is defined analogously and it can be derived under the dynamics in Eq.~\eqref{Edzdtgen} to get~\cite{gardiner1985handbook,trousdale2012impact,rosenbaum2017spatial,baker2019correlated,baker2020inference}
\begin{equation}\label{ECz1d}
\begin{aligned}
\widetilde C^{\vvec z}(f) 
&= [r(f)I -W]^{-1} \widetilde C^{\vvec x}(f)[r(f)I -W]^{-*}
\end{aligned}
\end{equation}
where 
\[
r(f)=1-2\pi \tau f  i
\] 
is scalar.

Eq. \eqref{ECz1d} quantifies the covariance structure of $\vvec z(t)$ at any given frequency mode, $f$, but we are often specifically interested in the zero-lag temporal covariance, 
\begin{equation}\label{EcovInt}
\overline C^{\vvec z}_{jk}=\cov(\vvec z_j(t),\vvec z_k(t))=\int_{-\infty}^\infty \widetilde C^{\vvec z}_{jk}(f)df.
\end{equation}
From Eq.~\eqref{ECz1d}, we therefore have
\begin{equation}\label{ECbar0}
\overline C^{\vvec z}= \int_{-\infty}^\infty [r(f)I -W]^{-1} \widetilde C^{\vvec x}(f)[r(f)I -W]^{-*}df
\end{equation}

If the timescale of fluctuations in $\vvec x(t)$ are much slower than the timescale, $\tau$, of network dynamics then $\widetilde C^{\vvec x}(f)\approx 0$ at frequencies much higher than $1/(2\pi \tau)$ (see Figure~\ref{SuppFigPSD} in Appendix~\ref{AppSuppFigures}) . Indeed, we can take this to be the definition of the statement that the fluctuations in $\vvec x(t)$ are much slower than $\tau$. Since $r(f)\approx 1$ whenever $f\ll 1/(2\pi \tau)$, we therefore have that $r(f)\approx 1$ whenever $\widetilde C^{\vvec x}(f)$ is not close to zero (see Figure~\ref{SuppFigPSD} in Appendix~\ref{AppSuppFigures}). 
As a result, the only parts of the integrand that contribute to the integral in Eq.~\eqref{ECbar0} are the low frequency components, $f\approx 0$. In this case, we can replace $r(f)$ with $1$ to obtain  
\begin{equation}\label{ECzbar}
\overline C^{\vvec z}\approx [I -W]^{-1}\overline C^{\vvec x} [I -W]^{-T}.
\end{equation}
Note that Eq.~\eqref{ECzbar} can also be derived directly from the quasi-steady state approximation $\vvec z(t)\approx [I-W]^{-1}\vvec x(t)$. Our assumption above that $\widetilde C^{\vvec x}(f)\approx 0$ for $f>1/\tau$ is essentially equivalent to the assumption that the quasi-steady state approximation is accurate. Either approach can be used to derive Eq.~\eqref{ECzbar}. In Appendix~\ref{Sfast} below, we provide more details about the frequency dependence of variability and how it relates to the speed at which $\vvec x(t)$ varies.

When principal component analysis is applied to $\vvec z(t)$, the variance explained by each principal component is given by the ordered list of eigenvalues of the covariance matrix, $\overline C^{\vvec z}$. From Eq.~\eqref{ECzbar}, we see that these eigenvalues are proportional to the squares of the singular values of the matrix 
\[
R=[I -W]^{-1}\sqrt{\overline C^{\vvec x}}.
\]
where $\sqrt{\overline C^{\vvec x}}$ is the matrix square root of the symmetric positive definite covariance matrix, $\overline C^{\vvec x}$ (distinct from the entry-wise square root in general).   
Therefore, the decay of the variance explained by each principal component of $\vvec z$ (as in the blue dots in Figure~\ref{F1}d) are described by the squared singular values of $R$. 

In the models we consider, each $\vvec x_j(t)$ is an i.i.d. stochastic process, so $\overline C^{\vvec x}=vI$ is a multiple of the identity where $v=\var(\vvec x_j(t))$ is the stationary covariance of $\vvec x_j(t)$. Therefore, for our models, 
\[
R=[I-W]^{-1}\sqrt v
\]
and 
\begin{equation*}
\overline C^{\vvec z}\approx [I -W]^{-1}[I -W]^{-T}v
\end{equation*}
where recall that $v>0$ is a scalar. 

In more general classes of models in which $\overline C^{\vvec x}$is not a multiple of the identity matrix, the dimensionality of $\vvec z(t)$ might be reduced when $\overline C^{\vvec x}$ is effectively low-rank (see  Figure~\ref{SuppFigLowDimInput} in Appendix~\ref{AppSuppFigures}). However, in this situation, low-dimensional dynamics are inherited from the input, not generated by interactions within the network. By restricting to high-dimensional $\vvec x(t)$, we are focused on the question of whether low-dimensional dynamics arise from network interactions.

In conclusion, when $\vvec x(t)$ is high-dimensional, the dimensionality of the dynamics of $\vvec z(t)$ is determined by the effective rank of the matrix 
\[
A=[I-W]^{-1}.
\]
More specifically, the variance explained by each principal component of $\vvec z(t)$ is proportional to the square of the singular values of $A$,
\[
\textrm{var. explained by $k$th PC of }\vvec z(t)\approx \sigma^2_{A,k}\,v
\]
where $\sigma_{A,k}$ is the $k$th singular value of $A$ (assuming singular values and principal components are sorted in decreasing order) and $v=\var(\vvec x_j(t))$ is the stationary variance of each $\vvec x_j(t)$. 

Therefore, $\vvec z(t)$ is low-dimensional whenever $A$ has a small number of large singular values. As above, we can use the fact that singular values commute with matrix inverses to write our conclusions in terms of 
\[
Q=I-W.
\]
Specifically, 
\begin{equation}\label{EkPC}
\textrm{var. explained by $k$th PC of }\vvec z(t)\approx \frac{v}{\sigma^2_{Q,N-k}}
\end{equation}
where $\sigma_{Q,N-k}$ is the $(N-k)$th singular value of $Q$, i.e., the $k$th from the last singular value.

To demonstrate these analytical results, we repeated the simulation from Figure~\ref{F1}d and added the predictions from Eq.~\eqref{EkPC} as red dots (Figure~\ref{SuppFigTheory}a in Appendix~\ref{AppSuppFigures}). Surprisingly, the theory did not closely match the simulations. We suspected that this was due to finite sampling: The simulation was performed over the time interval $t\in[0,T]$ where $T=5\times 10^3$ (for comparison, $\tau=1$ and the correlation timescale of $\vvec x(t)$ was $\tau_x=10$). We suspected that Eq.~\eqref{ECzbar} would be accurate when $\overline{C}^\vvec{x}$ is replace by the empirical covariance matrix of $\vvec x(t)$. Under this substitution, Eq.~\eqref{EkPC} would be replaced by 
\[
\textrm{var. explained by $k$th PC of }\vvec z(t)\approx \frac{1}{\sigma^2_{ \hat B,N-k}}
\]
where
\[
\hat B=[I-W]\sqrt{{\hat C}^\vvec{x}}^{-1}
\]
and ${\hat C}^\vvec{x}$ is the the empirical covariance matrix. 
Or, equivalently and more simply, 
\begin{equation}\label{EkR}
\textrm{var. explained by $k$th PC of }\vvec z(t)\approx \sigma^2_{\hat R,k}
\end{equation}
where
\[
\hat R=[I-W]^{-1}\sqrt{{\hat C}^\vvec{x}}
\]
is the sampled value of $R$. Using Eq.~\eqref{EkR} gives a  more accurate prediction (Figure~\ref{SuppFigTheory}a  in Appendix~\ref{AppSuppFigures}, green dots). This confirms that the error in the red dots from  Figure~\ref{SuppFigTheory}a is due largely to under-sampling of $\vvec x(t)$. We next increased the simulation time ten-fold to $T=5\times 10^4$. In this case, the original Eq.~\eqref{EkPC} was more accurate (Figure~\ref{SuppFigTheory}b  in Appendix~\ref{AppSuppFigures}, red dots), further confirming that the errors in   Figure~\ref{SuppFigTheory}a are due largely to sampling error.

As concluded previously, $\vvec z(t)$ exhibits low-rank suppression (as demonstrated by the last blue dot in Figure~\ref{F1}d) whenever $Q$ has a small number of asymptotically large singular values (equivalently, whenever $A$ has a small number of asymptotically small singular values). Now, we may conclude that $\vvec z(t)$ has low-dimensional dynamics whenever $Q$ has some asymptotically \textit{small} singular values (equivalently, whenever $A$ has a small number of asymptotically large singular values). 
Conversely, \textit{high-dimensional} dynamics can only occur whenever $Q$ does not have any asymptotically small singular values (equivalently, whenever $A$ does not have any asymptotically large singular values). 
We can summarize as follows
\begin{quotation}
\noindent Low-rank suppression (as in Figure~\ref{F1}g) occurs whenever $Q=I-W$ has at least one asymptotically large singular value.

\noindent High-dimensional dynamics (as in Figure~\ref{F1}d) can only occur whenever $Q=I-W$ lacks any asymptotically small singular values.
\end{quotation}

Previously, in Section~\ref{SLDD}, we effectively showed that $Q=I-W$ has large singular values (and therefore low-rank suppression occurs) whenever $W$ is strongly low-rank. Conditions under which $Q=I-W$ \textit{lacks small singular values} (and therefore high-dimensional dynamics occur) are not so simple. Specifically, the lack or presence of small singular values depends on the recurrent alignment matrix,
\[
P=V^TU,
\]
which measures the alignment between the left and right singular vectors, $\vvec u_k$ and $\vvec v_k$. Specifically, $P_{jk}=\vvec v_j\cdot \vvec u_k$ so that $|P_{jk}|=1$ whenever $\vvec v_k=\pm \vvec u_k$ and $P_{jk}=0$ whenever $\vvec v_k$ is orthogonal to $\vvec u_k$. Note that singular values of $P$ are bounded by unity, $\sigma_P\le 1$. We next show that if all singular values of $P$ are $\mathcal O(1)$ then the network exhibits high-dimensional responses to  high-dimensional inputs. 

\setcounter{claim}{0}
\begin{claim}
Under the model and assumptions 1--3 in the main text, if $P=V^TU$ does not have any asymptotically small singular values then 
the dynamics of $\vvec z(t)$ are high-dimensional. 
\end{claim}
\begin{proof}
\noindent We will prove this claim by proving its contrapositive:
\begin{quotation}
\noindent If the dynamics of $\vvec z(t)$ are low-dimensional then $P=V^TU$ has at least one asymptotically small singular value. 
\end{quotation}
The dynamics of $\vvec z(t)$ are {low} dimensional whenever $\vvec z(t)$ has a small number of dominant principal components as in Figure~\ref{F2}d  (conversely, $\vvec z(t)$ is high-dimensional whenever there is no such dominant principal component, as in Figure~\ref{F1}d). From the discussion above, we know that $\vvec z(t)$ is {low} dimensional whenever $A=[I-W]^{-1}$ has at least one asymptotically large singular value or, equivalently, whenever $Q=I-W$ has an asymptotically small singular value. Therefore, our original claim is equivalent to the following:
\begin{quotation}
\noindent If $Q=I-W$  has at least one asymptotically small singular value then $P=V^TU$ also has at least one asymptotically small singular value.
\end{quotation}
We will prove this version of the claim directly. 
Assume that $Q$ has an asymptotically small singular value. Then there is a $\vvec z$ satisfying $\|\vvec z\|=1$ and
\[
(I-W)\vvec z=o(1)
\]
where the notation $o(1)$ means that $\|(I-W)\vvec z\|\to 0$ as $N\to\infty$. 
It is sufficient to show that there is a $\vvec y$ with $\|\vvec y\|=1+o(1)$ satisfying 
\[
\|P\vvec y\|=o(1).
\]
We have that
\[
W_0\vvec z+W_1\vvec z=\vvec z+o(1)
\]
Multiplying both sides on the left by $U^T$ gives 
\[
\Sigma V^T\vvec z+U^TW_1\vvec z=U^T\vvec z+o(1)
\]
Since $W_1$ is random and independent from $W_0$ and $U$, and since $U^T$ projects from $N$ to $r\ll N$ dimensions, the term $U^TW_1\vvec z$ represents the projection of a random vector onto the low-dimensional column space of $U$, so $\|U^TW_1\vvec z\|=o(1)$. We therefore have
\[
\Sigma V^T\vvec z=U^T\vvec z+o(1).
\]
We next claim that $\vvec z=UU^T\vvec z+o(1)$. To see why this is true, note that $UU^T\vvec z$ is the orthogonal projection of $\vvec z$ onto the column space of $W_0$. Since $W_0$ dominates $W=W_0+W_1$ (and therefore the the column space of $W$ is dominated by that of $W_0$)  and $W\vvec z\approx \vvec z$, we may conclude that $\vvec z$ lies predominantly in the column space of $W$ and therefore of $W_0$. In other words, $\vvec z=UU^T\vvec z+o(1)$.  
Hence, we can rewrite the equality above as
\[
\Sigma V^T UU^T\vvec z=U^T\vvec z+o(1).
\]
which reduces to
\[
\Sigma P\vvec y=\vvec y+o(1).
\]
where $P=V^TU$ and $\vvec y = U^T\vvec z$. Note again that $\|\vvec y\|=\|U^T\vvec z\|=1+o(1)$ since $\|\vvec z\|=1$ and $\vvec z=UU^T\vvec z+o(1)$.  We then have that 
\[
\|P\vvec y\|=\|\Sigma^{-1}\vvec y\|+o(1)=o(1)
\]
because $\Sigma^{-1}$ is a diagonal matrix with $o(1)$ terms on the diagonal. 
\end{proof}

The proof of Claim~1 also tells us the dominant directions of variability whenever dynamics are \textit{low-dimensional} as in Figure~\ref{F2} and Figure~\ref{F3}g. {Since $\vvec z$ is approximately aligned to the $U$ in the proof, we may conclude that low-dimensional dynamics are caused by excess variability along the column space of $U$, \textit{i.e.}, the column space of $W_0$.  This explains the observations in Figure~\ref{F2}e. More specifically, low-dimensional dynamics arise when external input and the random part of internal input (\textit{i.e.}, $W_1\vvec z$) have variability that overlap with the column space of $V$. Because these right singular vectors are associated with large singular values of $W$, this  variability gets mapped to strong variability in the column space of $U$. Since $P$ has small singular values, this variability in the column space of $U$ does not overlap with the column space of $V$, so it cannot be cancelled by recurrent interactions. This creates an effectively feedforward mapping that amplifies variability along the directions $\vvec v_k$ and maps it to variability in the directions of $\vvec u_k$. This mechanism is a form of non-normal amplification~\cite{trefethen1991pseudospectra,murphy2009balanced,hennequin2012non} (see the discussion in Section~\ref{SLDD} for a more in-depth comparison).}

{These results can also help us understand some subtleties of low-rank suppression. Our theory from Section~\ref{SLRS} shows that the input $\vvec x_a=\vvec u_k+(W_1\vvec v_k+\vvec v_k)/\sigma_k$ is suppressed. Since $\sigma_k$ is large, this input is close to $\vvec u_k$, but has a correction term $\vvec x_c=(W_1\vvec v_k+\vvec v_k)/\sigma_k$. In Figure~\ref{F1}g, the ``aligned'' input that produced a suppressed response was chosen simply as $\vvec x=\vvec u$, without the correction term (using $\vvec x_a$ with the correction term also produced a suppressed response; not pictured). For the example in Figure~\ref{F2}, we used $\vvec x=\vvec x_a$ with the correction term because using $\vvec x=\vvec u$ without a correction term did not produce a suppressed response. To understand why this happens, note that networks like the one in Figure~\ref{F2} (for which $P=V^TU$ is singular or has a small singular value) take inputs from the directions $\vvec v_k$ and amplify them in proportion to $\sigma_k$ (according to our proof of Claim~1). Hence, the $\vvec v_k/\sigma_k$ term in $\vvec x_a$ is small by a factor of $\sigma_k$, but the response is also amplified by a factor of $\sigma_k$, so the correction term produces a $\mathcal O(1)$ response and cannot be ignored. In summary, the correction term, $\vvec x_c$, is necessary to include when $P$ is singular or has small singular values (as in Figure~\ref{F2}).}


{
\section{Dimensionality of dynamics when inputs are not slowly varying.}\label{Sfast}
}

{
In Section~\ref{SLDD} and Appendix~\ref{SStochasticAnalysis}, we assumed that $\vvec x(t)$ varied on a slower timescale than the timescale, $\tau$, of the network dynamics, justifying the quasi-steady state approximation in Eq.~\eqref{Equasi}. While this assumption simplifies our analysis and conclusions, we can understand the dimensionality of dynamics without it. Here, we analyze the dimensionality of $\vvec z(t)$ under all of the conditions from Assumptions 1--3 except the assumption that $\vvec x(t)$ is slowly varying. 
}

{
Beginning with Eq.~\eqref{ECbar0}, we see that the dimensionality of $\vvec z(t)$ in general depends on the interaction between $\widetilde C^{\vvec x}(f)$, $W$ and $r(f)$ across different frequencies, $f$. Here, we again focus on the special case that each $\vvec x_j(t)$ is i.i.d. In this case, $\widetilde C^{\vvec x}(f)=c_x(f)I$ where $c_x(f)$ is the scalar power spectral density of each $\vvec x_j(t)$. Therefore, the covariance matrix of $\vvec z(t)$ can be written as~\cite{gardiner1985handbook}
\begin{equation}\label{ECzfint}
\overline C^{\vvec z}=\int_{-\infty}^\infty [r(f) I-W]^{-1}[r(f)I-W]^{-*} c_x(f)df
\end{equation}
where $\cdot^{-*}$ is the inverse conjugate transpose and $r(f)=1-2\pi \tau f i $.} 

{
First consider $\widetilde C^{\vvec z}(f)$ at frequencies lower than the timescale of the dynamics, $f\ll 1/(2\pi \tau)$. At these frequencies, we have that $r(f)\approx 1$ so that 
\[
\widetilde C^{\vvec z}(f)\approx [I-W]^{-1}[I-W]^{-T}c_x(f)
\]
at sufficiently low frequencies. 
Since $c_x(f)$ is a scalar, we  have that 
\[
\overline C^{\vvec z}\propto [I-W]^{-1}[I-W]^{-T}
\]
for $f\ll 1/(2\pi \tau)$. Therefore, the arguments and conclusions from Section~\ref{SLDD}, Appendix~\ref{SStochasticAnalysis}, and Claim~1 apply in general to the dimensionality of dynamics at low frequencies. Specifically, the dimensionality of the dynamics at low frequencies (\textit{i.e.}, the slow dynamics) can only be low-dimensional when $P=V^TU$ has a small singular value. This is true even when $\vvec x(t)$  has power at high frequencies. 
}

{
However, the stationary covariance matrix of $\vvec z(t)$ accounts for variability at all frequencies (see Eqs.~\eqref{ECzfint}).
When $\vvec x(t)$ is slowly varying, $c_x(f)\approx 0$ at higher frequencies (see Figure~\ref{SuppFigPSD}), so the low-frequency analysis above is sufficient (consistent with the results obtained in Appendix~\ref{SStochasticAnalysis} and Section~\ref{SLDD}). 
When $\vvec x(t)$ has power at high frequencies, these higher frequencies need to be accounted for. 
}

{
At high frequencies satisfying $f\gg \|W\|/(2\pi \tau)$, the $r(f)I$ terms in Eq.~\eqref{ECzfint} dominates the $W$ terms so that 
\[
\widetilde C^{\vvec z}(f)\approx \frac{c_x(f)}{1+4\pi^2 f^2 \tau^2}I
\]
which is a multiple of the identity and therefore full-rank.  Hence, when $\vvec x(t)$ has variability at these high frequencies (\textit{i.e.}, when $c_x(f)$ is not close to zero for $f\gg \|W\|/(2\pi \tau)$; see Figure~\ref{SuppFigFastx}a), these high frequency components can only add high-dimensional variability to $\vvec z(t)$ because $\vvec z(t)$ tracks $\vvec x(t)$ at high frequencies. 
}

{
If $\vvec x(t)$ has variability \textit{only} at these high frequencies ($f\gg \|W\|/(2\pi \tau)$), then only these frequencies would contribute to the integral in Eq.~\eqref{ECzfint} and therefore $\vvec z(t)$ would be high-dimensional regardless of the structure of $W$. However, it is somewhat rare that natural stochastic processes have power only at higher frequencies. 
}

{
More commonly, when $\vvec x(t)$ varies quickly, it has power at both low and high frequencies. In other words, $c_x(f)$ is far from zero for both small and large $f$. For example, white noise has equal power at all frequencies and Orenstein Uhlenbeck processes have more power at low than high frequencies even when they have fast timescales. 
When $\vvec x(t)$ has power at low and high frequencies, the integral in Eq.~\eqref{ECzfint} has contributions from both low and high frequencies. Low-dimensional variability can only be introduced at low frequencies and, by Claim~1 and our arguments above, can only be introduced when $P=V^TU$ has a small singular value. Since these low-dimensional dynamics represent an asymptotically large singular value, this combination of low-dimensional variability at low frequencies and high-dimensional variability at high frequencies results in low-dimensional variability overall since the low-dimensional variability is strong while the high-dimensional variability is $\mathcal O(1)$. As a result, when $\vvec x(t)$ has power at both low and high frequencies, the conclusions of Claim~1 apply. Specifically, low-dimensional variability arises only when $P$ has small singular values. 
To test this conclusion, we repeated the simulations from Figure~\ref{F3} with timescales swapped. Specifically, we set $\tau=10$ and $\tau_x=1$ so that $c_x(f)$ is far from zero at both low frequencies ($f\ll 1/(2\pi \tau)$) and high frequencies ($f\gg \|W\|/(2\pi \tau)$; see Figure~\ref{SuppFigFastx}a  in Appendix~\ref{AppSuppFigures}). As predicted, low-dimensional variability only arises in the case that $P=V^TU$ has a singular value at zero (Figure~\ref{SuppFigFastx}b--f), identical to the case when $\vvec x(t)$ varies slowly (compare Figure~\ref{SuppFigFastx} to Figure~\ref{F3}).
}

~

\section{Additional Simulations and Data}\label{AppSuppFigures}

This appendix contains results from additional simulations and data. The simulations themselves and data are described in the respective captions and contextualized within the main text and Appendix.

 \begin{figure*}[h]
 \centering{
 \includegraphics[width=6.5in]{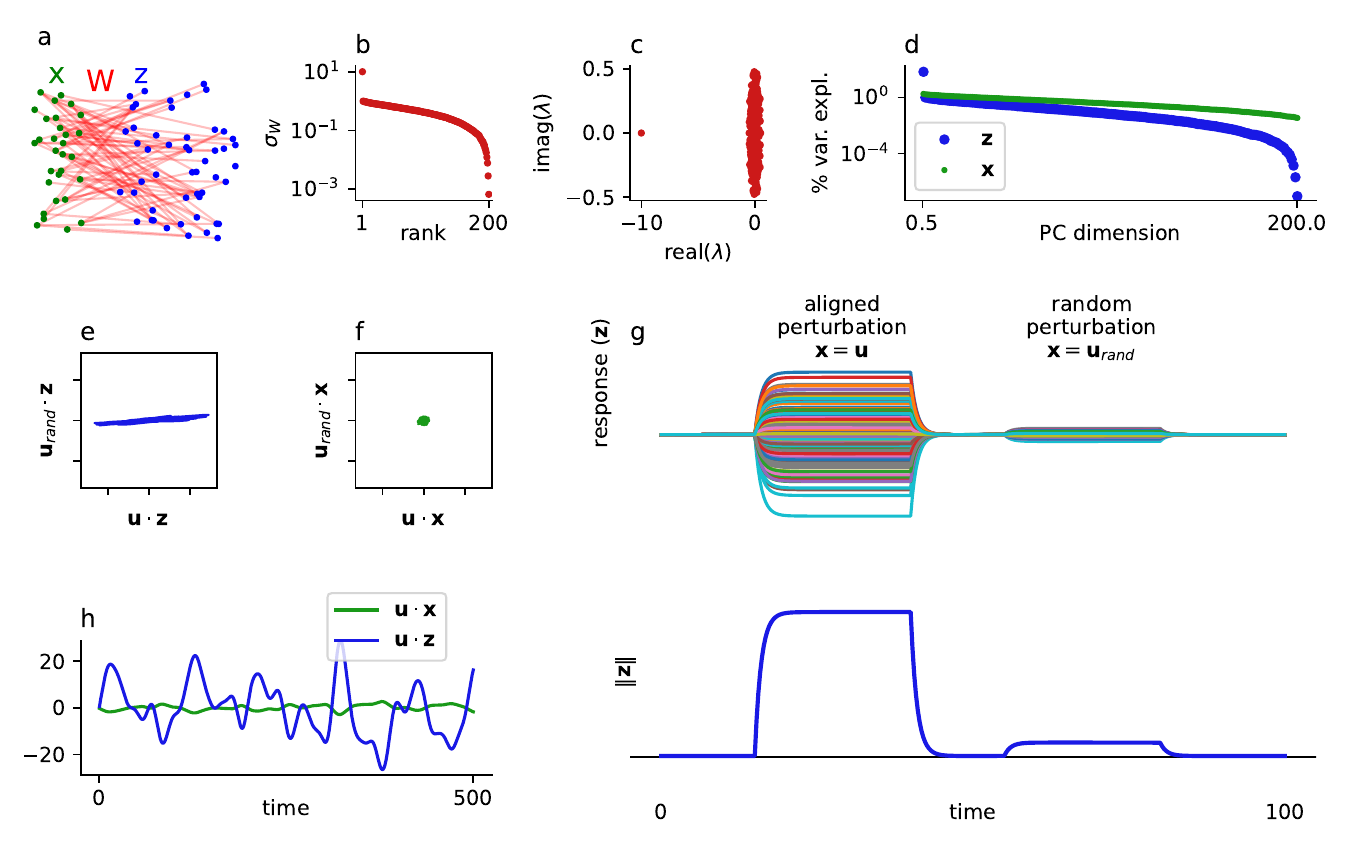}
 }
 \caption{{\bf Response properties of a feedforward network with rank-one structure.} Same as Figure~\ref{F1} except the recurrent network was replace by a feedforward network with dynamics satisfying $\tau d\vvec z/dt=-\vvec z'+W\vvec x$. Unlike the recurrent network in Figure~\ref{F1}, the feedforward network is most sensitive to inputs aligned with its low-rank structure, and its dynamics are dominated by one-dimensional variability. }
\label{SuppFigFfwd}
 \end{figure*}

 \begin{figure*}[h]
 \centering{
 \includegraphics[width=3in]{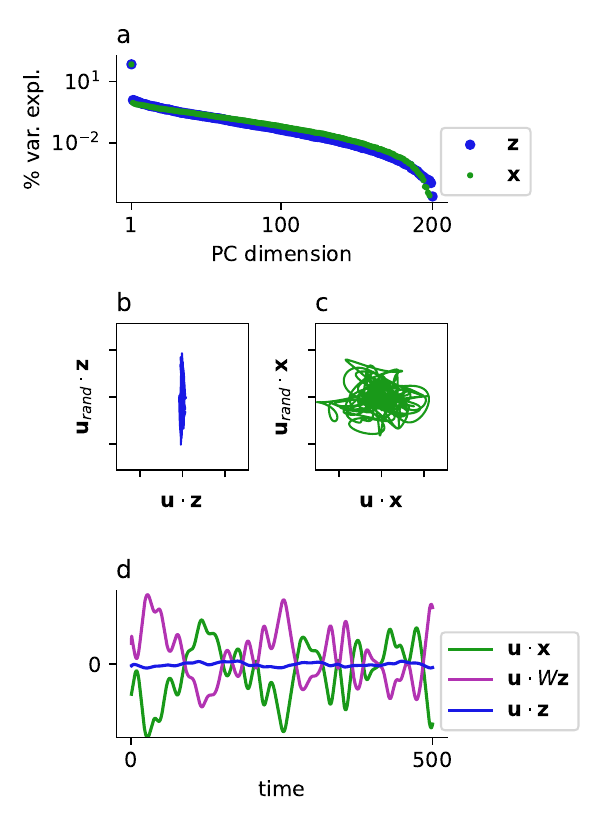}
 }
 \caption{{\bf Dynamics of a network with rank-one structure driven by an input perturbation with one dimensional structure.} Same as Figure~\ref{F1} except the dynamics obeyed $\tau d\vvec z/dt=-\vvec z+W\vvec z+W_x \boldsymbol \eta(t)$ where $\boldsymbol \eta(t)$ is a realization of the same Gaussian stochastic process used for $\vvec x(t)$ in Figure~\ref{F1}, and $W_x$ is an effectively low-rank matrix generated identically to, but independently from $W$. The ``effective input'', $\vvec x(t)=W_x \boldsymbol \eta(t)$, is therefore low-dimensional (see green dots in panel a). The network dynamics, $\vvec z(t)$, inherit low-dimensional dynamics from $\vvec x(t)$ (blue dots in panel a), but low-rank suppression and cancellation are still exhibited along the vector, $\vvec u$, defining the low-rank structure of $W$ (panels b--d).}
\label{SuppFigLowDimInput}
 \end{figure*}

 \begin{figure*}[h]
 \centering{
 \includegraphics[width=6.5in]{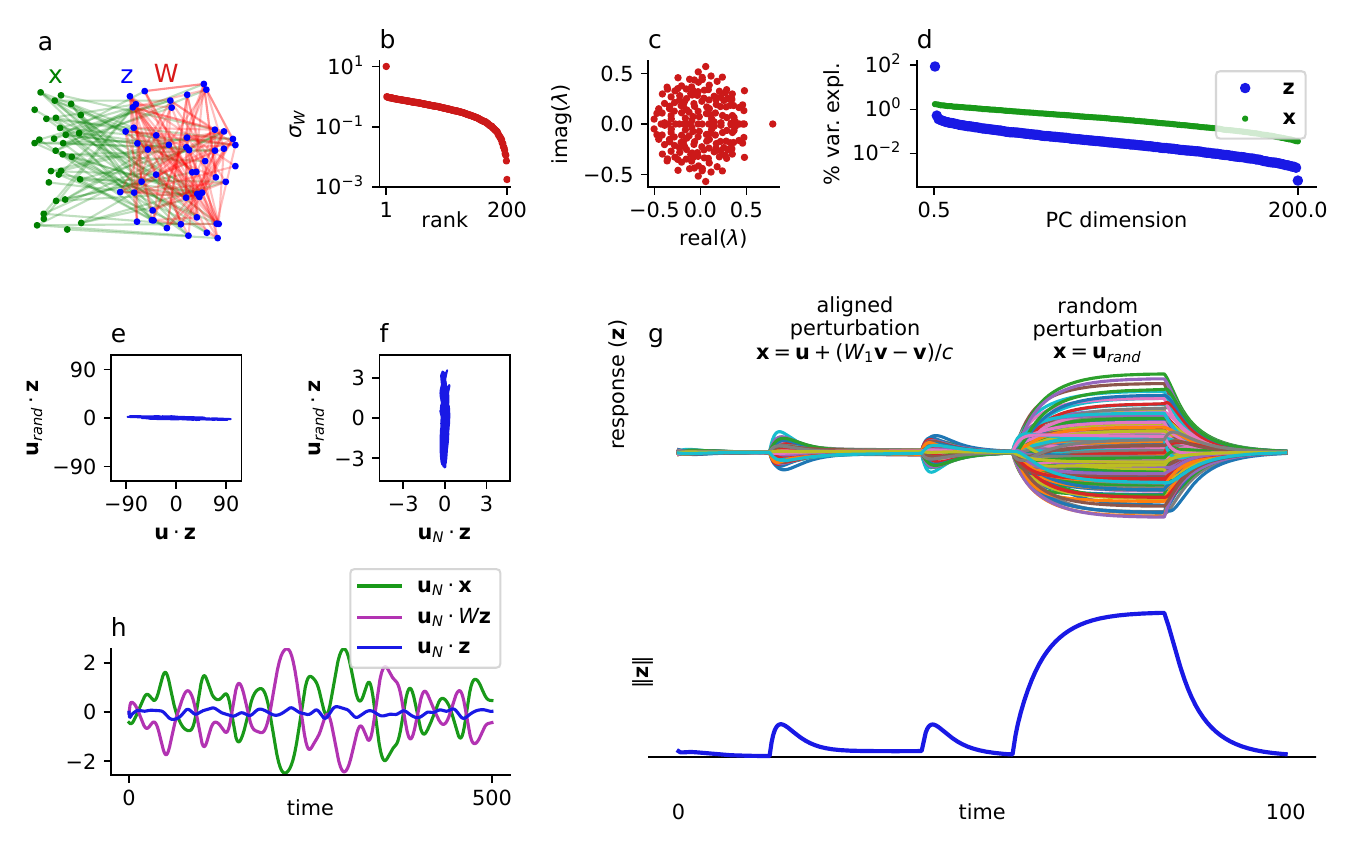}
 }
 \caption{{\bf Response properties of a network with independent  left  and right singular vectors.} Same as Figure~\ref{F2} except $W_0=c\vvec u\vvec v^T$ where $\vvec u$ and $\vvec v$ are independent random unit vectors. The network behaves similarly to the network in Figure~\ref{F2} because $\vvec u$ and $\vvec v$ are nearly orthogonal.}
\label{SuppFigNonNormal}
 \end{figure*}

 \begin{figure*}[h]
 \centering{
 \includegraphics[width=6.5in]{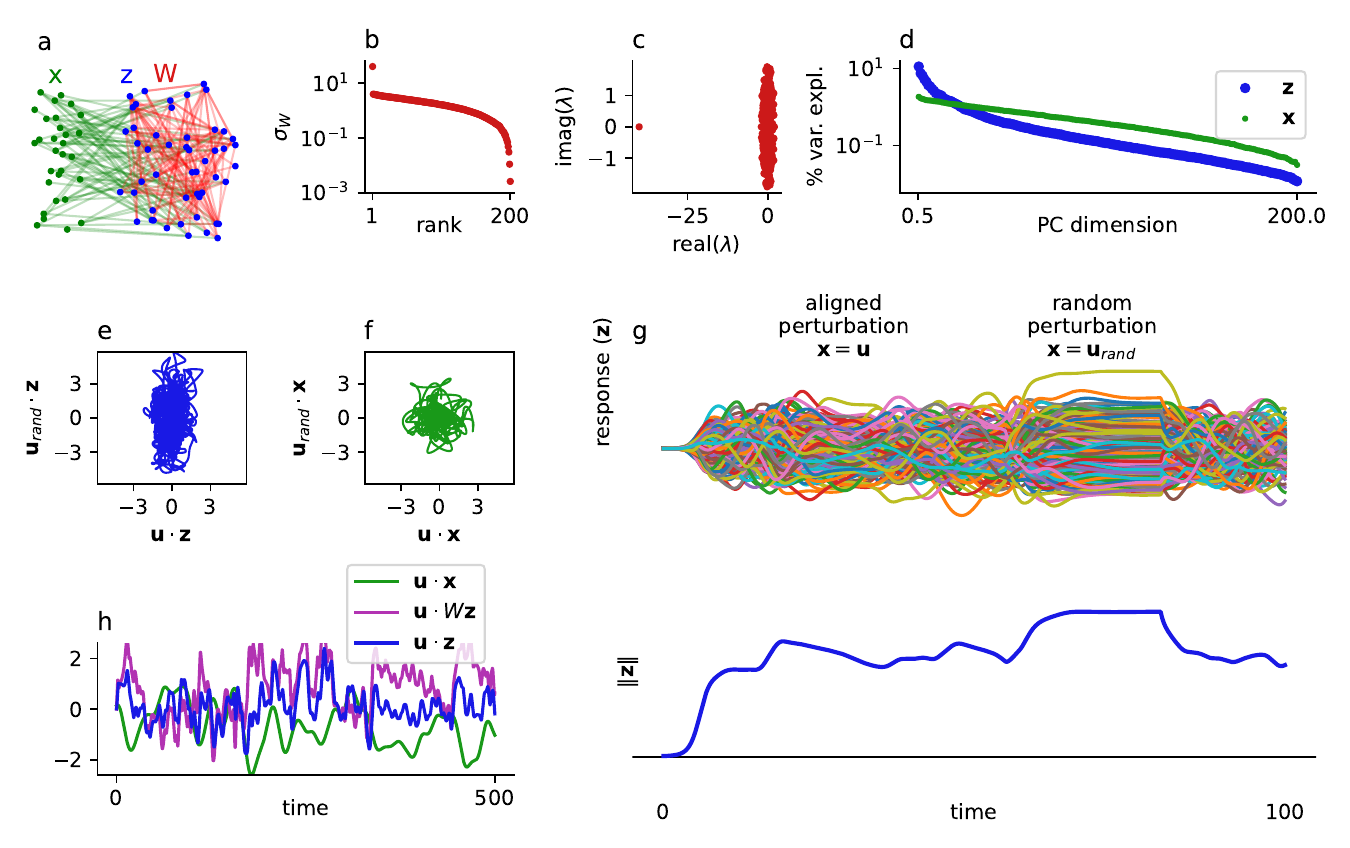}
 }
 \caption{{\bf Dynamics of an unstable, chaotic network with rank-one structure.} Same as Figure~\ref{F1} except $\rho=2$ and the magnitude of $c$ was increased by the same factor ($c=-40$) to maintain the same degree of low-rank structure (the same ratio between the spectral radius of the random part, $\rho$, and the magnitude of the low-rank part, $|c|$), and the dynamics obey Eq.~\eqref{EdzdtNonLin}. The instability produced by taking $\rho>1$ generates chaotic dynamics.}
\label{SuppFigUnstable}
 \end{figure*}

 \begin{figure*}[h]
 \centering{
 \includegraphics[width=6.5in]{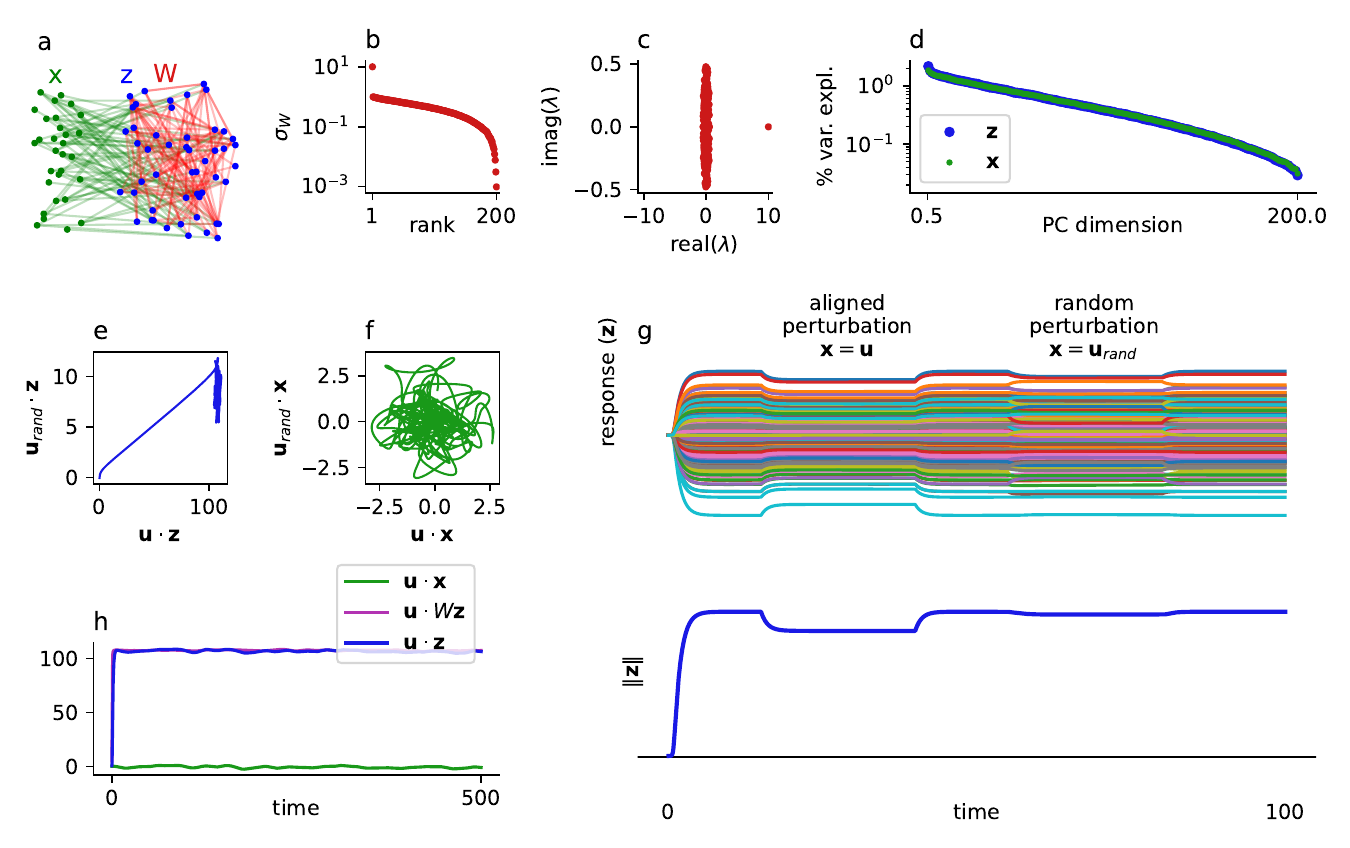}
 }
 \caption{{\bf Dynamics of an unstable, non-chaotic network with rank-one structure.} Same as Figure~\ref{F1} except $c=10$ and the dynamics obey Eq.~\eqref{EdzdtNonLin}. The fixed point at $\vvec z=0$ in unstable because $c>1$.}
\label{SuppFigUnstable2}
 \end{figure*}

 \begin{figure*}[h]
 \centering{
 \includegraphics[width=6.5in]{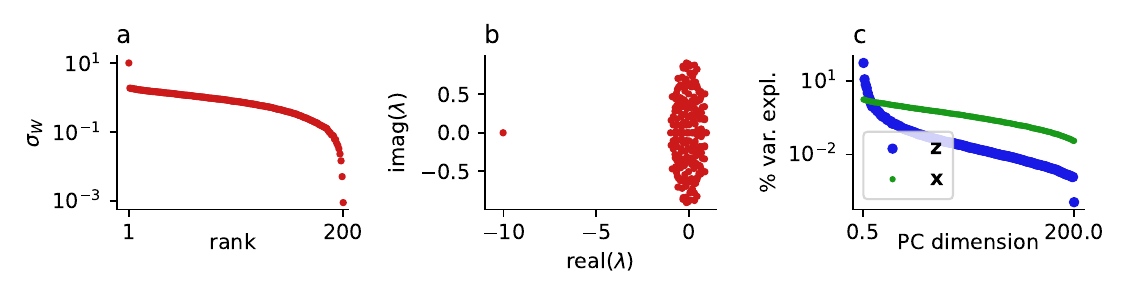}
 }
 \caption{{{\bf Dynamics of a stable network  with rank-one structure close to criticality.} Same as Figure~\ref{F1} except $\rho=0.95$. The fixed point at $\vvec z=0$ is stable, but close to instability because the maximum real part of the eigenvalues of the Jacobian is close to $0$. }}
\label{SuppFigCritical}
 \end{figure*}

 \begin{figure*}
 \centering{
 \includegraphics[width=3in]{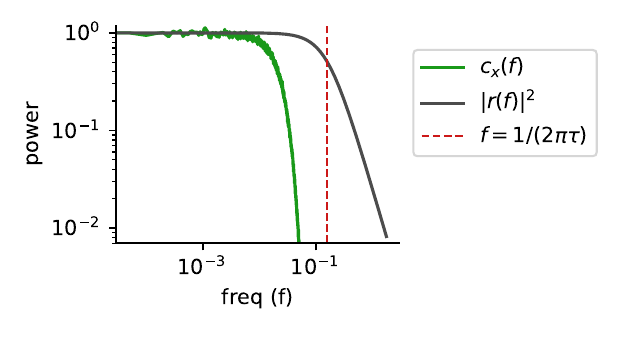}
 }
 \caption{{{\bf Power spectral density of $\vvec x(t)$ used in simulations.} Average power spectral density of $\vvec x_j(t)$ (green) compared to $|r(f)|^2=1/(1+4\pi^2 f^2 \tau^2)$ (gray) when $\tau_x=10$ and $\tau=1$, as in all figures except for Figure~\ref{SuppFigFastx}. Vertical red line shows $f=1/(2\pi \tau)$.}}
 \label{SuppFigPSD}
 \end{figure*}

 \begin{figure*}
 \centering{
 \includegraphics[width=5in]{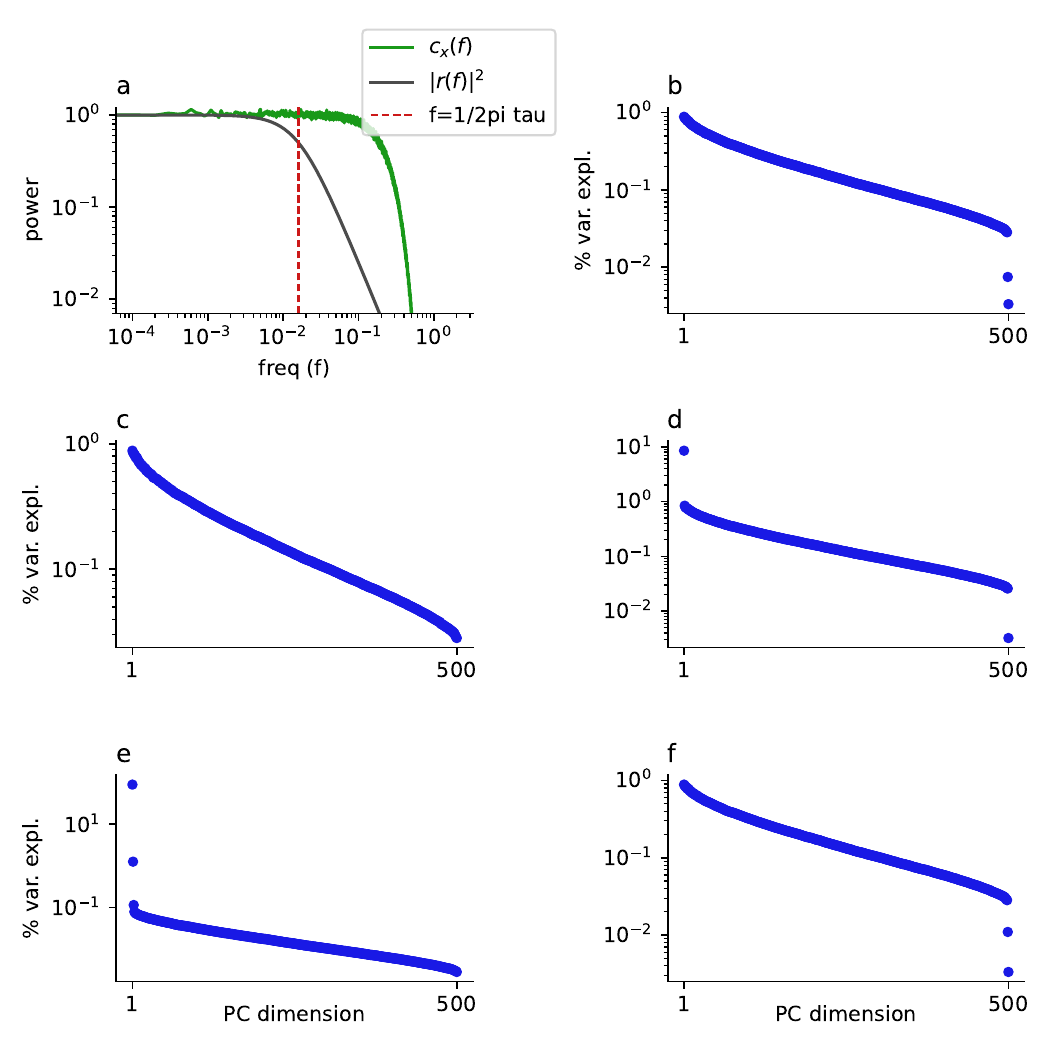}
 }
 \caption{{{\bf High- and low-dimensional dynamics when inputs are not slow.} Simulations from Figure~\ref{F3} were repeated with $\tau_x=1$ and $\tau=10$. {\bf a)} Average power spectral density of $\vvec x_j(t)$ (green) compared to $|r(f)|^2=1/(1+4\pi^2 f^2 \tau^2)$ (gray). Vertical red line shows $f=1/(2\pi \tau)$. Since $\vvec x(t)$ is faster than intrinsic network dynamics ($\tau_x<\tau$), $c_x(f)$ has power at higher frequencies where $r(f)\approx 0$. 
 {\bf b-e)} Same as Figure~\ref{F3} panels c,f,i,l,o respectively except that we set $\tau_x=1$ and $\tau=10$. Variability in $\vvec z(t)$ is high- and low-dimensional under the same conditions.}}
 \label{SuppFigFastx}
 \end{figure*}

 \begin{figure*}[h]
 \centering{
 \includegraphics[width=6.5in]{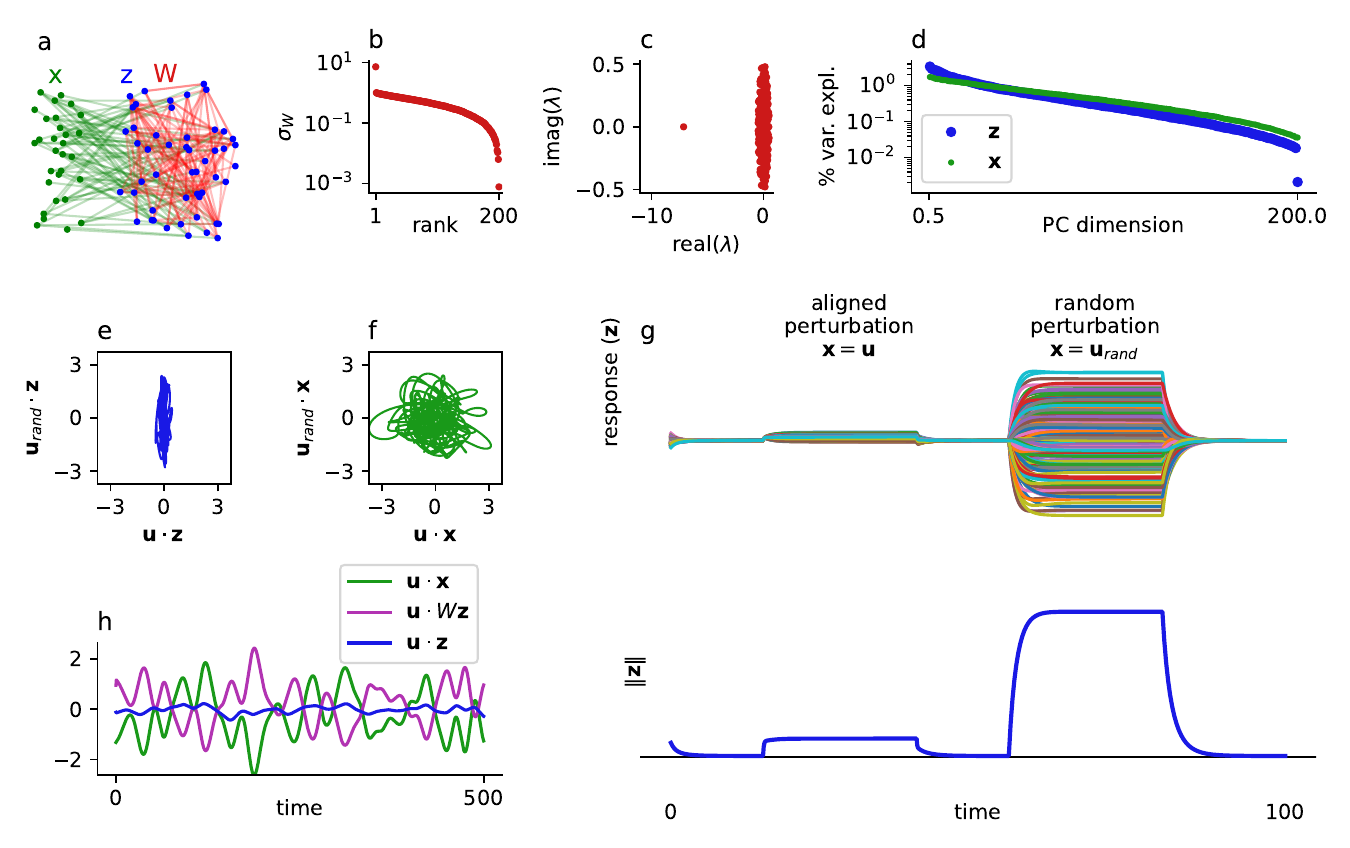}
 }
 \caption{{\bf Response properties of a network with biased weights.} Same as Figure~\ref{F1} except we used the dynamics in Eq.~\eqref{EdzdtNonLin}, and $W_{jk}$ are drawn i.i.d. from a normal distribution with mean $-0.5/\sqrt N$ and standard deviation $0.5/\sqrt N$ where $N=200$. }
\label{SuppFigBias}
 \end{figure*}

 \begin{figure*}[h]
 \centering{
\includegraphics[width=3in]{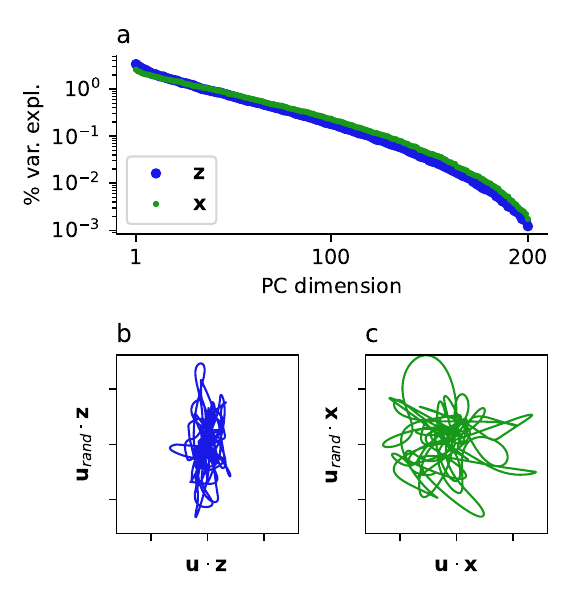}
 }
 \caption{{\bf Low-dimensional dynamics in a modular network.} Same as Figure~\ref{F1} except we used the dynamics in Eq.~\eqref{EdzdtNonLin} and  the modular network from Figure~\ref{F4}.}
\label{SuppFigModular}
 \end{figure*}

 \begin{figure*}[h]
 \centering{
\includegraphics[width=4in]{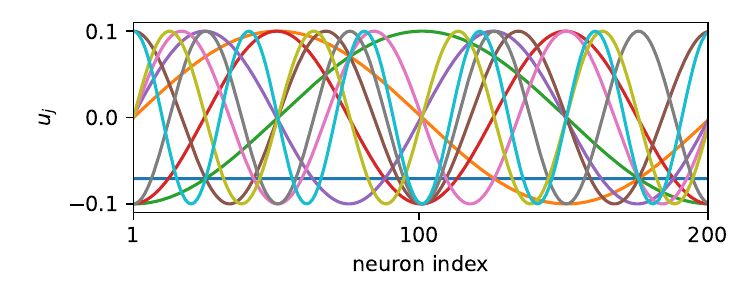}
 }
 \caption{{\bf Dominant singular vectors of a spatial network are Fourier modes.} The singular vectors corresponding to the ten largest singular values of the network in Figure~\ref{F5}.}
\label{SuppFigFourier}
 \end{figure*}

 \begin{figure*}[h]
 \centering{
 \includegraphics[width=3in]{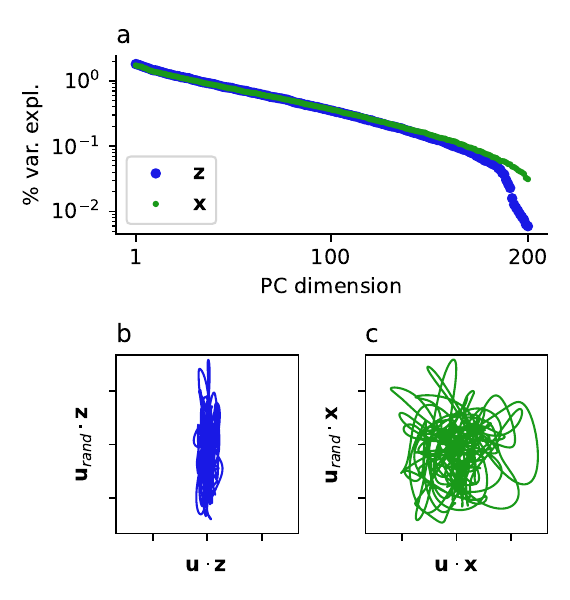}
 }
 \caption{{\bf Low-dimensional dynamics in a spatial network.} Same as Figure~\ref{F1} except we used the dynamics in Eq.~\eqref{EdzdtNonLin} and the spatial network from Figure~\ref{F5}.}
\label{SuppFigSpace}
 \end{figure*}

 \begin{figure*}[h]
 \centering{
 \includegraphics[width=4in]{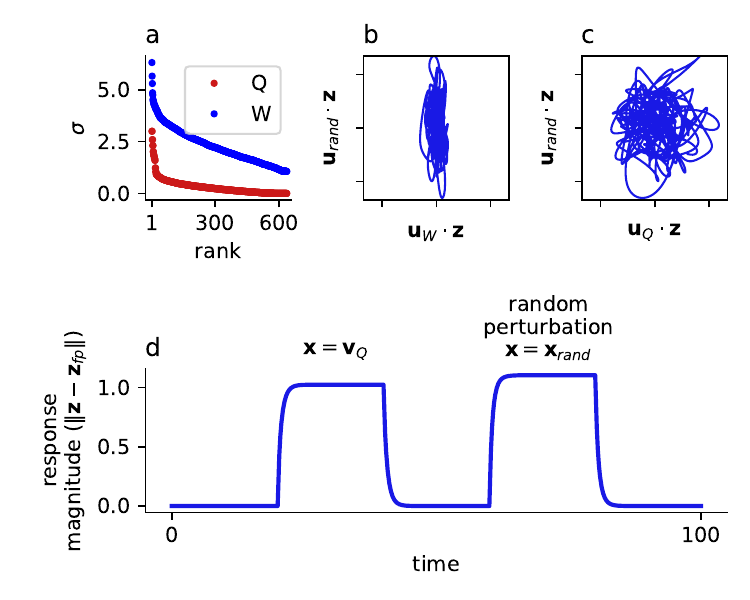}
 }
 \caption{{{\bf Singular values and dynamics in an epidemiological network.} {\bf a)} Singular values of $Q$ (red) and $W$ (blue). {\bf b)} Projections of $\vvec z$ for the model in Figure~\ref{F6}. {\bf c)} Same as Figure~\ref{F6}a except the dominant singular vector of $Q$ was used to define the aligned direction instead of the dominant singular value of $W$. The lack of suppression demonstrates the necessity of using the effective connectivity matrix, $W$, in place of $Q$, which is the connectivity matrix used to define the dynamics.}}
\label{SuppFigEpidemiological}
 \end{figure*}

 \begin{figure*}[h]
 \centering{
 \includegraphics[width=6.5in]{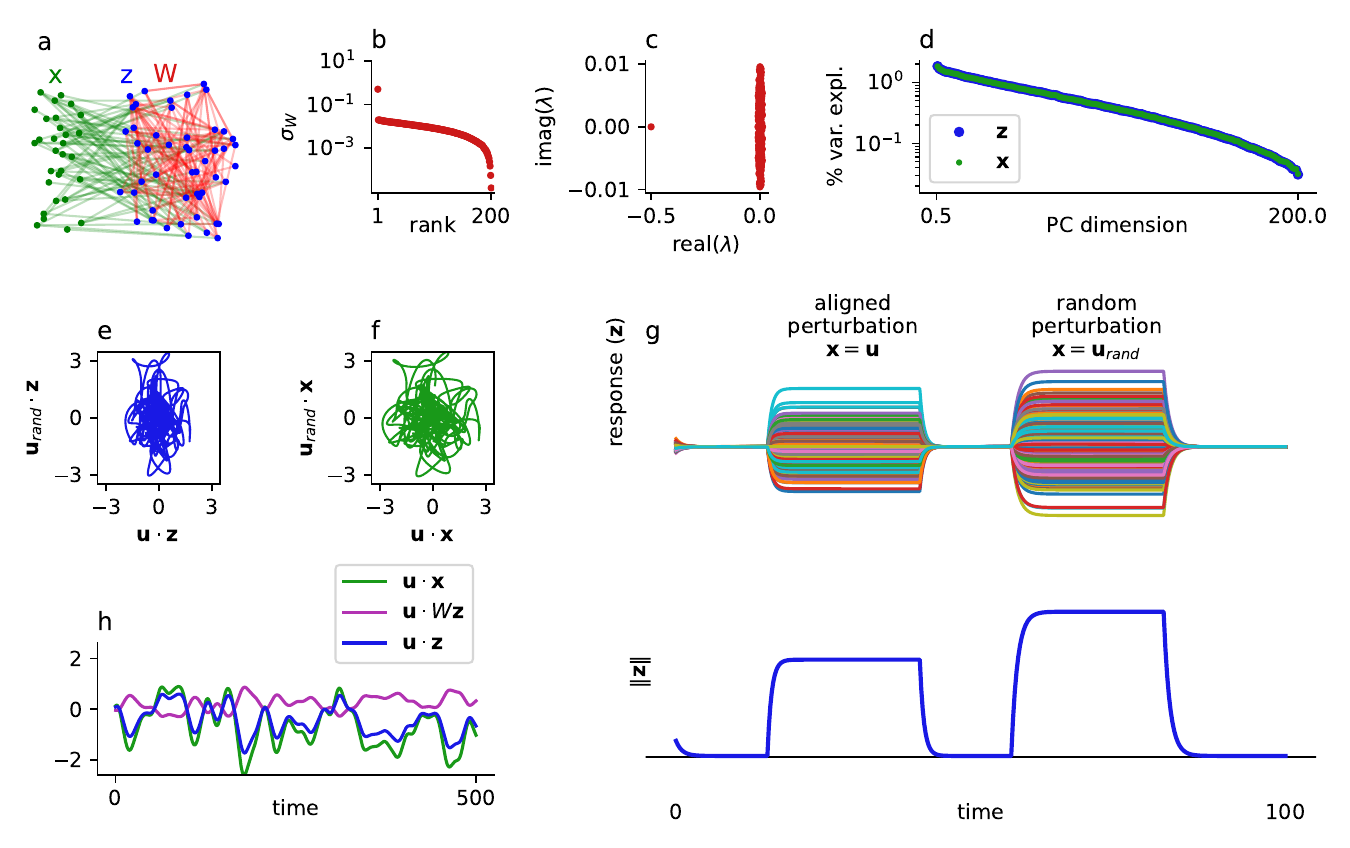}
 }
 \caption{{\bf Response properties of a weakly low-rank  network with $\rho\ll 1$ and  $c<0$.} Same as Figure~\ref{F1} except $c=-0.5$ and $\rho=0.01$. }
\label{SuppFigWeak}
 \end{figure*}

 \begin{figure*}[h]
 \centering{
 \includegraphics[width=6.5in]{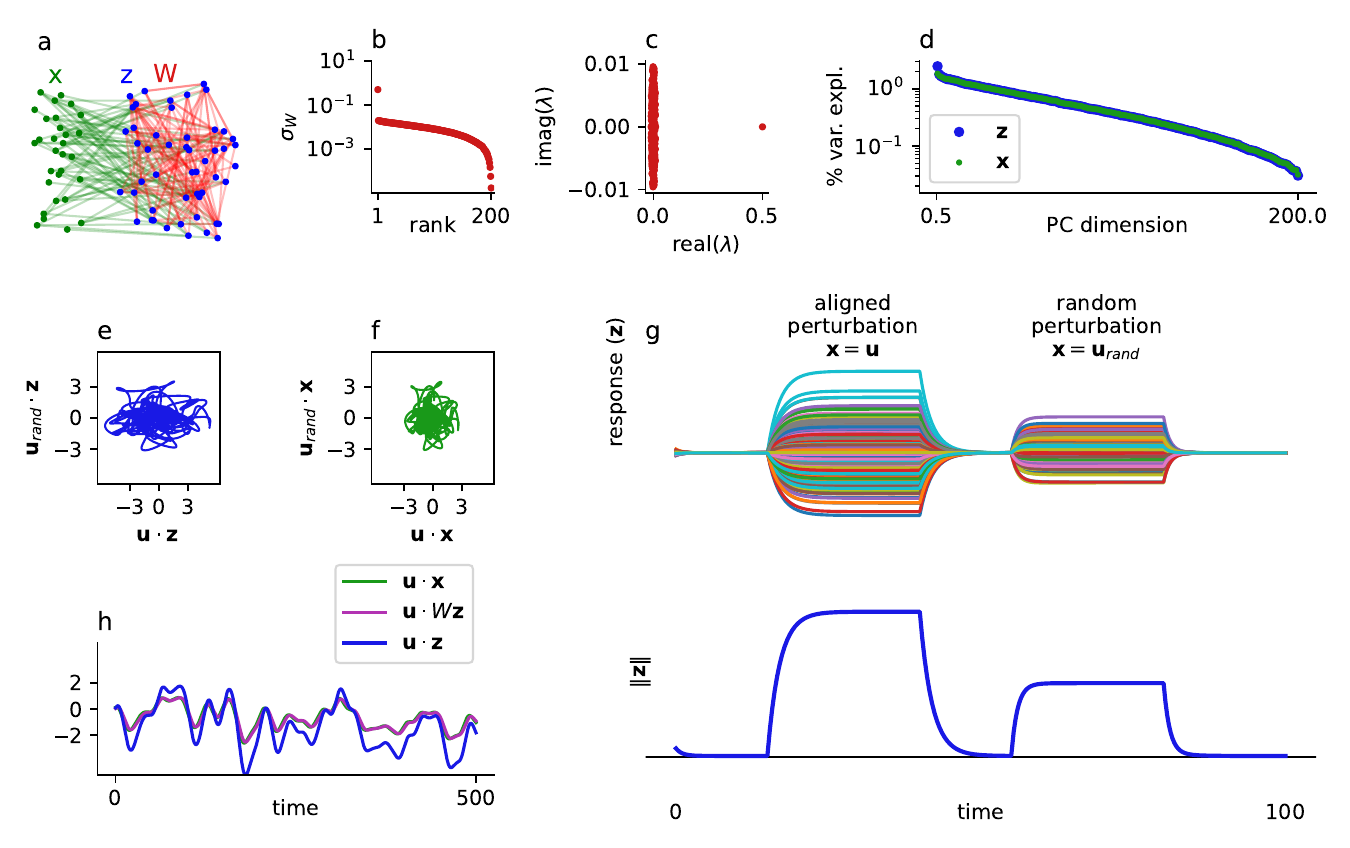}
 }
 \caption{{\bf Response properties of a weakly low-rank  network with $\rho\ll 1$ and $c>0$.} Same as Figure~\ref{F1} except $c=0.5$ and $\rho=0.01$. Note that stability requires that $c<1$ when $c>0$.}
\label{SuppFigWeak2}
 \end{figure*}

 \begin{figure*}[h]
 \centering{
 \includegraphics[width=2.5in]{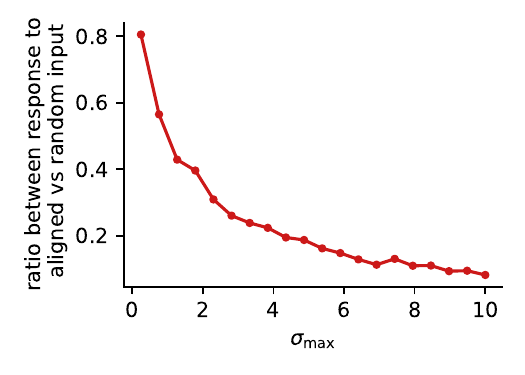}
 }
 \caption{{{\bf Quantifying low-rank suppression as the maximum singular value is scaled.} We repeated the simulation from Figure~\ref{F1} while changing the value of $c<0$ to change the value of the maximum singular value of $W$, $\sigma_\textrm{max}=|c|$. This demonstrates a smooth transition between weakly and strongly low-rank structure. For each value of $c$, we computed the ratio of $\|\vvec z\|$ in response to an aligned input versus a random input, \textit{i.e.}, the height of the two bumps in Figure~\ref{F1}g as a measure of low-rank suppression. Low-rank suppression occurs when this ratio is small. }}
\label{SuppFigscalec}
 \end{figure*}

 \begin{figure*}[h]
 \centering{
 \includegraphics[width=6.5in]{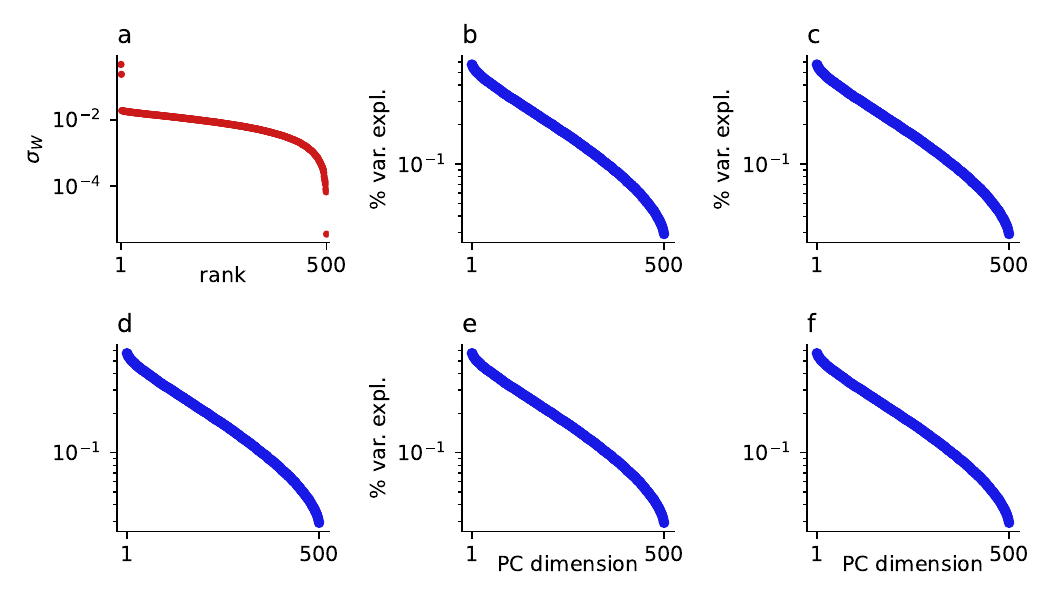}
 }
 \caption{{{\bf High-dimensional dynamics in weakly low-rank networks.} {\bf a)} Singular values for a weakly low-rank network in which $\rho=0.01$, $\sigma_1=0.5$, and $\sigma_2=0.25$. 
 {\bf b-f)} Same as Figure~\ref{F3}c,f,i,l,o except that $\rho=0.01$, , $\sigma_1=0.5$, and $\sigma_2=0.25$. The network response is high-dimensional in every case.}}
\label{SuppFigWeak3}
 \end{figure*}

 \begin{figure*}
 \centering{
 \includegraphics[width=5.5in]{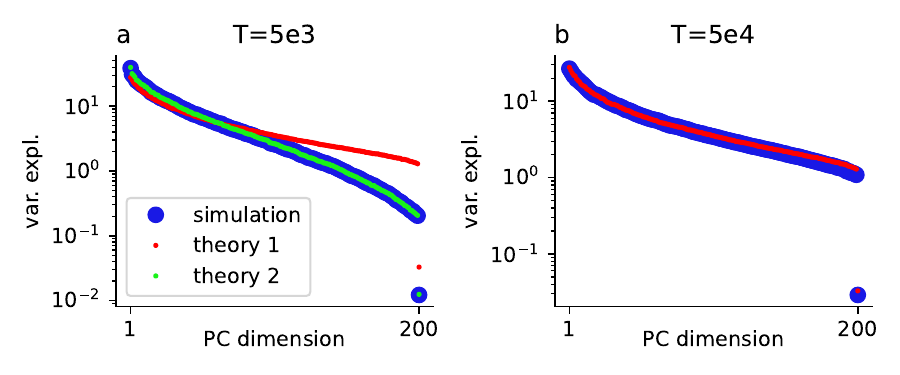}
 }
 \caption{{\bf Comparison between theory and simulations for shorter and longer simulations.} {\bf a)} Same as Figure~\ref{F1}d except we added the theoretical values form Eq.~\eqref{EkPC} (red dots) and the values obtained  from Eq.~\eqref{EkR} (green dots). The network was simulated for $T=5\times 10^3$ time units. {\bf b)} Same as b except we increased the simulation time to $T=5\times 10^4$.   }
\label{SuppFigTheory}
 \end{figure*}
 
 \clearpage


%

\end{document}